\documentclass[journal,twoside,web]{ieeecolor}

\usepackage{generic}
\usepackage{cite}
\usepackage{amsmath,amssymb,amsfonts}
\usepackage{algorithmic}
\usepackage{graphicx}
\usepackage{booktabs}
\usepackage{makecell}
\usepackage{array}
\usepackage{tabularx}
\usepackage{algorithm,algorithmic}
\usepackage{epstopdf}
\usepackage{hyperref}
\usepackage{dsfont}
\hypersetup{hidelinks}
\DeclareMathOperator{\col}{col}
\DeclareMathOperator{\diag}{diag}
\usepackage{textcomp}
\newtheorem{thm}{Theorem}[section]
\newtheorem{corollary}{Corollary}[section]
\newtheorem{lem}{Lemma}[section]

\newtheorem{definition}{Definition}[section]

\newtheorem{assum}{Assumption}[section]
\newtheorem{rem}{Remark}[section]

\def\BibTeX{{\rm B\kern-.05em{\sc i\kern-.025em b}\kern-.08em
    T\kern-.1667em\lower.7ex\hbox{E}\kern-.125emX}}
\begin{document}
\title{Continuous-Time Distributed Seeking for Variational Generalized Nash Equilibrium of Online Game}
\author{Jianing Chen, Sichen Qian, Chuangyin Dang, \IEEEmembership{Senior Member, IEEE} and Sitian Qin$^{*}$
\thanks{© 202X IEEE. Personal use of this material is permitted. Permission from IEEE must be obtained for all other uses, in any current or future media, including reprinting/republishing this material for advertising or promotional purposes, creating new collective works, for resale or redistribution to servers or lists, or reuse of any copyrighted component of this work in other works.}
\thanks{ Jianing Chen is with the Department of Mathematics, Harbin Institute of Technology, Weihai, 264209, China, and also with the Department of Systems Engineering, City University of Hong Kong, Hong Kong (e-mail: jchen2265-c@my.cityu.edu.hk).}
\thanks{Sichen Qian is with the School of Mathematics, Southeast University, Nanjing, 210096, China (e-mail:qiansichen@foxmail.com).}
\thanks{Chuangyin Dang is with the Department of Systems Engineering, City University of Hong Kong, Hong Kong (e-mail: mecdang@cityu.edu.hk).}
\thanks{Sitian Qin is with the Department of Mathematics, Harbin Institute of Technology, Weihai, 264209, China (e-mail: qinsitian@163.com).}
}
\maketitle

\begin{abstract}
This paper mainly investigates a class of distributed Variational Generalized Nash Equilibrium (VGNE) seeking problems for both online noncooperative games and online aggregative games with time-varying coupling inequality constraints. Two novel continuous-time distributed VGNE seeking algorithms are proposed, which realize the constant regret bound and sublinear fit bound, superior to those of the criteria for online optimization problems and online games. Furthermore, to reduce unnecessary communication among players, a dynamic event-triggered mechanism involving internal variables is introduced into the distributed VGNE seeking algorithm, while the constant regret bound and sublinear fit bound are still maintained. Also, the Zeno behavior is strictly prohibited. Moreover, we further investigate the impact of communication noise on the player's measurement of its neighbors' relative states. It is demonstrated that both the regret and fit bounds remain valid as long as the noise level is not excessively large. This result reveals, to some extent, the proposed algorithm's noise-resilient capability. Finally, an online Uncrewed Aerial Vehicle (UAV) swarm game and an online Nash-Cournot game are given to demonstrate the validity of the theoretical results.
\end{abstract}


\section{Introduction}
\IEEEPARstart{N}{oncooperative} games, where self-interested intelligent players aim to optimize their individual cost functions, arise in various engineering applications, such as congestion control \cite{yin2011nash}, plug-in electric vehicles \cite{shi2023distributed}, and smart grids \cite{ma2022fully}. A distinguishing characteristic of noncooperative games, as opposed to standard optimization problems, lies in the coupling of players' cost functions, where a single point minimizing all players' costs simultaneously may not exist. Building on the pioneering work of John Nash \cite{nash1950bargaining}, a fundamental equilibrium concept, namely the Nash Equilibrium (NE), was proposed. At an NE, no player can decrease its own cost by unilaterally changing its strategy, given the strategies of all other players. Furthermore, when the players' actions are coupled through coupling constraints, a more general concept, referred to as the Generalized Nash Equilibrium (GNE), was introduced to characterize the equilibrium behavior under such coupling \cite{facchinei2010generalized}. However, a general GNE is highly complex to search for due to the interdependence of players' actions. Thus, there is a growing interest in identifying a subset of GNEs that are more tractable. In this regard, a specific class of GNEs, known as VGNEs, have been introduced and extensively studied, owing to their equivalence to variational inequality problems and their improved economic interpretation that enhances fairness \cite{facchinei2007generalized}. To compute a VGNE algorithmically, information on other players must be communicated, observed, or measured. However, owing to the competitive nature of noncooperative games, full knowledge of all players' actions may be impractical. Consequently, estimates based on distributed consensus among players have garnered significant interest, as they eliminate the need for a dedicated inter-player communication infrastructure \cite{ye2017distributed,gadjov2019distributed}. 

Due to the inherent uncertainty and complexity of the environment, cost functions and constraints in noncooperative games are typically time-varying. This variability introduces additional challenges in developing accurate theoretical models and applying classical algorithms, such as those in \cite{Wang2020,Li2024,Yang2024,Zeng2019,Ding2024}, for VGNE computation. Typically, in these dynamic settings, online distributed algorithms for online games have emerged as a prominent approach for sequential decision-making over the past two decades \cite{Xu2022,Meng2023,Meng2021,Deng2023,Lu2020,Belgioioso2024,Lu2024}. For example, in \cite{Xu2022}, an online distributed algorithm based on average consensus technique was designed for $N$-cluster games, achieving a sublinear regret bound. However, it is noteworthy that most existing research on VGNE seeking in online games is limited to discrete-time settings.\begin{table*}[htbp]
	\centering
	\caption{Comparison of the proposed algorithm herein with related works measured by regret and fit.}
	\begin{tabular}{p{0.15\textwidth} p{0.1\textwidth} p{0.2\textwidth} p{0.45\textwidth}}
		\hline
		\toprule
		Reference & Algorithm Type & Game Type & Regret and Fit Bounds\\
		\midrule
		Xu et al. \cite{Xu2022}& Discrete-time & $N$-cluster game &$\mathcal{R}^{\top}=\mathcal{O}\left(1+\sqrt{T}\right)$ \\
		\midrule
Meng et al.\cite{Meng2023}&Discrete-time&Noncooperative game with coupling inequality constraints
		&$\mathcal{R}^{\top}=\mathcal{O}\left(T^{\frac{13}{14}}\right)$, $\mathcal{F}^{\top}=\mathcal{O}\left(T^{\frac{13}{14}}\right)$\\
		\midrule
		Meng et al. \cite{Meng2021} & Discrete-time &Noncooperative game with coupling inequality constraints and private set constraints&$\mathcal{R}^{\top}=\mathcal{O}\left(T^{\max\left\{\frac{1}{2}+a_{1},\frac{1}{2}+\frac{a_{2}}{2},1-\frac{a_{1}}{2}\right\}}\right)$,

 $\mathcal{F}^{\top}=\mathcal{O}\left(T^{\max\left\{1-\frac{a_{1}}{2},2-\frac{3a_{2}}{2},\frac{1}{2}+\frac{a_{1}}{2},\frac{3}{2}+a_{1}-\frac{3a_{2}}{2}\right\}}\right)$\\
		\midrule
		Deng and Zuo \cite{Deng2023}&Discrete-time&Noncooperative game with coupling inequality constraints and private set constraints &$\mathcal{R}^{\top}=\mathcal{O}\left(\sqrt{T\left(\frac{v_{T}+1}{\alpha_{T}^{2}}+\sum_{t=1}^{T}\alpha_{t}\right)}\right)$,

$\mathcal{F}^{\top}=\mathcal{O}\left(\frac{1}{\alpha_{T}}\sqrt{\frac{v_{T}+1}{\alpha_{T}^{2}+\sum_{t=1}^{T}\alpha_{t}}}\right)$\\
		\midrule
		Lu et al.\cite{Lu2020} & Discrete-time &Noncooperative game with coupling inequality constraints and private set constraints&$\mathcal{R}^{\top} =\mathcal{O}\left(\sqrt{T\left(\frac{\Theta_{T}+1}{\gamma^{2}(T)}+\sum_{t=1}^{T}\gamma(t)\right)}\right)$,

$\mathcal{F}^{\top}=\mathcal{O}\Big(\frac{\sqrt{\left(\frac{\Theta_{T}+1}{\gamma^{2}(T)}+\sum_{t=1}^{T}\gamma(t)\right)\left(1+\sum_{t=1}^{T}\gamma^{2}(t)\right)}}{\gamma(T)}\Big)$\\
		\midrule
		Herein &Continuous-time&Noncooperative game with coupling inequality constraints and private set constraints& Continuous-time communication:
					$\mathcal{R}^{\top}=\mathcal{O}\left(1\right)$,~$\mathcal{F}^{\top}=\mathcal{O}\left(\sqrt{T}\right)$

		Event-triggered communication:
				$\mathcal{R}^{\top}=\mathcal{O}\left(1\right)$,~$\mathcal{F}^{\top}=\mathcal{O}\left(\sqrt{T}\right)$\\
		\bottomrule
		\hline
	\end{tabular}
	\label{tab:example}
\end{table*}The performance of these discrete-time algorithms often hinges on the careful and sometimes unpredictable design of step sizes (learning rates). In contrast, continuous-time algorithms, which avoid the requirement for such precise step size design, can offer a powerful alternative for VGNE seeking. Moreover, many real-world systems operate in continuous time, such as the continuous dynamics of electric current flow, making continuous-time algorithms particularly useful for modeling such systems. Whether an effective continuous-time distributed algorithm exists for online games remains an open research question.

While the use of communication networks can be advantageous for the design of distributed algorithms, practical limitations on communication resources often arise. To significantly reduce communication costs, one effective strategy is to employ an event-triggered mechanism, wherein players exchange information with their neighbors only when certain predefined triggering conditions are activated \cite{cai2023distributed}. Distributed event-triggered NE seeking algorithms have been developed for various scenarios, including aggregative games \cite{Shi2019} and general noncooperative games \cite{zhang2021nash,xu2022hybrid}. For instance, an adaptive event-triggered mechanism was developed for distributed NE seeking in noncooperative games \cite{xu2022hybrid}. However, all these studies primarily focused on static noncooperative games. To the best of our knowledge, there is limited research addressing online games within an event-triggered framework. Besides, most existing work on event-triggered mechanisms is applicable only to NE seeking. In scenarios involving coupling constraints, VGNE seeking algorithms require the transmission of information related to Lagrangian multipliers. In the context of event-triggered communication for online games with time-varying coupling constraints, players often need to share their private information, including actions, estimations and Lagrangian multipliers, with their neighbors. This requirement could potentially increase the frequency of information exchanges among players. Therefore, it is still an important challenging issue to design an efficient event-triggered mechanism for distributed VGNE seeking in online games with time-varying coupling constraints.

The main contributions of this paper can be highlighted as follows (see TABLE \ref{tab:example} for the comparison with existing relevant works).
\begin{enumerate}
	\item As far as we know, this paper is the first to investigate continuous-time distributed VGNE seeking for both online noncooperative games and online aggregative games with time-varying coupling constraints. Unlike the discrete-time VGNE seeking algorithms discussed in \cite{Xu2022,Meng2023,Meng2021,Deng2023,Lu2020,Belgioioso2024,Lu2024}, the continuous-time algorithms eliminate the need for precise step size tuning and provide a more natural framework for modeling real-world systems. It is demonstrated that the proposed algorithms achieve a constant regret bound and an $\mathcal{O}\left(\sqrt{T}\right)$ fit bound. In contrast, most existing distributed VGNE seeking algorithms \cite{Xu2022,Meng2023,Meng2021,Deng2023,Lu2020,Belgioioso2024,Lu2024} typically attain only sublinear regret bounds. Moreover, by employing the average tracking technique, the parameter conditions required for the online aggregative game are further relaxed.
	\item To reduce the communication cost among players, this paper further introduces a dynamic event-triggered mechanism into the continuous-time online VGNE seeking algorithm,
which is proven to still guarantee the constant regret bound and $\mathcal{O}\left(\sqrt{T}\right)$ fit bound in the case of aperiodic discrete communication. Besides the established regret and fit bounds here, compared with the static event-triggered mechanism in \cite{Shi2019,yuan2017event}, the triggering signal of the dynamic event-triggered mechanism can be dynamically adjusted as required.
\item Furthermore, to the best of our knowledge, the paper is the first to study the influence of measurement noise on the VGNE seeking process for online games, whereas existing studies \cite{cao2021decentralized,yu2024distributed} investigate this aspect only in online optimization problems with uncoupled cost functions. It is further shown that, provided the noise level remains within a moderate range, the proposed algorithm still guarantees a constant regret bound and a sublinear fit bound. It demonstrates the algorithm's inherent resilience to measurement noise.
\end{enumerate}
\par
The rest of this paper is organized as follows. Table \ref{biaono} summarizes the notations used in this paper. Section \ref{PF} presents a brief introduction to graph and operator theories and formulates the online noncooperative game and the online aggregative game with two performance metrics. In Section~\ref{MR}, two continuous-time distributed VGNE seeking algorithms are proposed for the online noncooperative game and the online aggregative game, respectively. In addition, one event-triggered and one noise-resilient VGNE seeking algorithm are developed for the online noncooperative game, together with the corresponding theoretical analyses. Section \ref{NE} displays an online UAV swarm game and an online Nash-Cournot game to testify the efficiency of the theoretical results. Conclusion for the whole paper is arranged in Section \ref{C}.
\par
\begin{table}[htbp] 
	\centering 
	\caption{List of notations.}\label{biaono} 
	\begin{tabularx}{\columnwidth}{lX}
		\toprule
		\textbf{Notations} & \textbf{Description}  \\
		\midrule
	$\mathbb{R},~\mathbb{R}^p\text{ and }\mathbb{R}^{p\times q}$ & Sets of real numbers, $p$-dimensional vectors and $p\times q$ matrices, respectively  \\ 
$\mathbb{R}_{\geqslant0}^{q}$&Set of nonnegative $p$-dimensional vectors\\
$(\cdot)^{\top},~\Vert\cdot\Vert$ and $\Vert\cdot\Vert_{1}$ & The transpose, the Euclidean norm and the $1$-norm, respectively\\
	$\diag\{\xi_{1},\xi_{2},\cdots,\xi_{N}\}$&A diagonal matrix whose principal diagonal elements are $\xi_{i}$ for $i=1,2,\cdots,N$\\
	$\col(x_1,x_2,\cdots,x_N)$& $[x_1^{\top},x_2^{\top},\cdots,x_N^{\top}]^{\top}\in\mathbb{R}^p$ with $x_i\in\mathbb{R}^{p_i}$ and $p=\sum_{i=1}^Np_i$\\
	$\mathbf{0}_{p\times q},~ \mathbf{1}_{p\times q}$&$p\times q$ matrices whose elements are all $0$ and $1$, respectively\\
	$\mathbf{I}_{p}$&A $p\times p$ identity matrix\\
	$\left[x\right]_{+}$& $\max\{\mathbf{0},x\}$\\
	$\otimes,~\langle \cdot,\cdot\rangle$ and $\times$& The Kronecker product, the inner product and the Cartesian product, respectively\\
	$\partial_{x} f(t,x)$&The partial derivative of function $f(t,x): \mathbb{R}\times\mathbb{R}^p\rightarrow\mathbb{R}$ on $x$\\
	$\lambda_{\min}(\mathcal{L}),~\lambda_{2}(\mathcal{L})$&The smallest and the second smallest eigenvalues of matrix $\mathcal{L}$, respectively\\
	$\operatorname{sign}\left(x_i\right)$&A signum function, where $\operatorname{sign}\left(x_i\right)=-1$ if $x_i<0 ; \operatorname{sign}\left(x_i\right)=1$ if $x_i>0 ; \operatorname{sign}\left(x_i\right)=0$ otherwise\\
	$\operatorname{sgn}(x)$&$\left(\operatorname{sign}\left(x_1\right),  \cdots, \operatorname{sign}\left(x_p\right)\right)^\top$, a component-wise signum function of a vector $x \in \mathbb{R}^p$\\
	$\vert \mathcal{S}\vert$&The number of elements in set $\mathcal{S}$\\
	$\text{cl}(\cdot)$& The closure of a set\\
	$f=\mathcal{O}(g)$&For functions $f:\mathbb{R}^{p}\rightarrow\mathbb{R}$ and $g:\mathbb{R}^{p}\rightarrow\mathbb{R}$, there exists a constant $C>0$ such that $\vert f(x)\vert\leqslant C\vert g(x)\vert$ for all $x\in\mathbb{R}^{p}$\\
$\mathcal{G}$& $\mathcal{(V,E)}$, an undirected graph with a node set $\mathcal{V=}$\{1,~2,~$\cdots$,~$\mathit{N}$\} and an edge set $\mathcal{E}$\\
$\mathcal{N}_{i}$ & $\{\mathit{j:(j,i)}\in\mathcal{E}\}$, the neighbors of node $\mathit{i}$\\
$\mathcal{A}$& $\{a_{ij}\}\in\mathbb{R}^{N\times N}$, the adjacency matrix, where $a_{ii}=0$, $a_{ij}>0$ if $j\in\mathcal{N}_i$ and else, $a_{ij}=0$\\
$d_i$ & $\sum^N_{j=1}a_{ij}$, the degree of node $i$\\
$\mathcal{L}$ & $\mathcal{D}-\mathcal{A}$, the Laplacian matrix with $\mathcal{D}=\diag\{d_1,d_2,\cdots,d_N\}$\\
$P_S$& The projection on a closed convex set $S$\\
$\mathbf{T}_S$&The tangent cone of a closed convex set $S$\\
$\Pi_S\left[x,v\right]$& $P_{\mathbf{T}_S(x)}(v)=\lim _{\delta \rightarrow 0^{+}} \frac{P_S(x+\delta v)-x}{\delta}$, the projection on the tangent cone of $S$ at $x$\\
$K_f,~K_g$& The upper bound of $\Vert J_{i}(t,x_{i},x_{-i})\Vert$ and $\Vert g_{i}(t,x_{i})\Vert$, respectively\\
$\Gamma$ & $\left\{\mathcal{V},J_{i},\Omega_{i}\cap U(t)\right\}$, a constrained online noncooperative game with cost function $J_i$, set constraint $\Omega_i$ and coupling constraint $U(t)$\\
$\Upsilon^i_j$& $i$th player's estimation on $j$th player's action\\ 
$F$ & $\text{col}(\partial_{x_1} J_{1}(t,x_{1},x_{-1}),\partial_{x_2} J_{2}(t,x_{2},x_{-2})\cdots$ $,\partial_{x_N} J_{N}(t,x_{N},x_{-N}))$, the pseudo-gradient\\
$\mathbf{F}$ &$\text{col}(\partial_{x_{1}}J_{1}(t,x_{1},\Upsilon^{1}_{-1}),\partial_{x_{2}}J_{2}(t,x_{2},\Upsilon^{2}_{-2}),$ $\cdots,\partial_{x_{N}}J_{N}(t,x_{N},\Upsilon^{N}_{-N}))$, the extended pseudo-gradient\\
$f_{i}(t,x_{i},\sigma(x))$&The $i$th player's cost function for an aggregative game with $\sigma(x)=\frac{1}{N}\sum_{i=1}^{N}\vartheta_{i}(x_{i})$ constituted by linear mappings $\vartheta_{i}(x_{i}):\mathbb{R}^{p_i}\rightarrow\mathbb{R}^n$\\
$\mathcal{R}^{\top}$& $\int_{0}^{T}\sum_i(J_{i}(t,x_{i}(t),x_{-i}^{*})-J_{i}(t,x_{i}^{*},x_{-i}^{*}))\text{d}t$, the static regret\\
$\mathcal{F}^{\top}$&$\|[\int_{0}^{T}\sum_{i}g_{i}(t,x_{i})\text{d}t]_{+}\|$, the fit\\
$k,~K_\mu,~\omega$& Positive constants\\
$R_{i},~S_i$&Selection matrices, see Page 5 for detailed definitions\\
$\Phi_{i}(t,x_{i},\phi_{i})$&$(\partial_{x_i}f_i(t,\cdot,\sigma)+\frac{1}{N}\partial_\sigma f_i^\top(t,x_i,\cdot)\partial\vartheta_i)|_{\sigma=\phi_i}$\\
$x_i,~\mu_i,~\psi_i,~\phi_i$& Player $i$'s action, primal-dual variable, auxiliary variable, estimation variable for $\sigma$, respectively\\
$m,~m_0,~l,~l_0$& Strongly monotone constant for $F$, minimal strongly convex constant for $J_i$, Lipschitz constant for $F$ and Lipschitz constant for $\mathbf{F}$\\
		\bottomrule
	\end{tabularx}
\end{table}
\section{Preliminaries and Problem Formulation}\label{PF}
\subsection{Graph Theory}
An undirected graph $\mathcal{G:=}\mathcal{(V,E)}$ is made up of a node set $\mathcal{V:=}$\{1,~2,~$\cdots$,~$\mathit{N}$\} and an edge set $\mathcal{E}$. $\mathcal{N}_{i}:=\{\mathit{j:(j,i)}\in$ $\mathcal{E}$\} stands for the neighbors of node $\mathit{i}$, where ($\mathit{j,i})\in\mathcal{E}$ means that node $\mathit{i}$ can communicate with node $\mathit{j}$. A path is a sequence of isolated vertices such that any pair of vertices appearing consecutively is an edge of graph $\mathcal{G}$. A graph $\mathcal{G}$ is said to be connected if there is a path between any two nodes. $\mathcal{A}=\{a_{ij}\}\in\mathbb{R}^{N\times N}$ is the adjacency matrix, where $a_{ii}=0$, $a_{ij}>0$ if $j\in\mathcal{N}_i$ and else, $a_{ij}=0$. Denote $d_i=\sum^N_{j=1}a_{ij}$ as the degree of node $i$. The Laplacian matrix of $\mathcal{G}$ is denoted by $\mathcal{L}=\mathcal{D}-\mathcal{A}$, where $\mathcal{D}=\diag\{d_1,d_2,\cdots,d_N\}$.
\begin{lem}\cite{guo2020predefined}\label{yinli11}
For a connected undirected graph $\mathcal{G}$, $\mathcal{L}$ is positive semidefinite, whose eigenvalues have order as $\lambda_1=0<\lambda_2 \leqslant \cdots \leqslant \lambda_N$. Also, $\mathcal{L} \mathbf{1}_N=\mathbf{0}_N$ and $\mathbf{1}_N^{\top} \mathcal{L}=\mathbf{0}_N^{\top}$.
\end{lem}
\subsection{Operator Theory}
\begin{definition}\cite{Zhu1991} A few basic definitions on convex analysis are displayed,\begin{itemize}
		\item If for any $x_{1}, x_{2} \in \mathbb{R}^{p},k \in[0,1],$ there is $f(k x_{1}+(1-k)x_{2})\leqslant k f(x_{1})+(1-k) f(x_{2}),$ then $f:\mathbb{R}^{p}\rightarrow\mathbb{R}$ is called convex.
		\item If for any $x_{1}, x_{2} \in \mathbb{R}^{p},k \in[0,1],$ there is $f(k x_{1}+(1-k)x_{2})\leqslant k f(x_{1})+(1-k) f(x_{2})-\frac{m}{2}k(1-k)\Vert x_{1}-x_{2}\Vert^{2}$, then $f:\mathbb{R}^{p}\rightarrow\mathbb{R}$ is called $m$-strongly convex.
\item If there exists a constant $\theta>0$ such that $\left\| f(x_{1})-f(x_{2})\right\|\leqslant\theta\Vert x_{1}-x_{2} \Vert$, for any $x_{1}, x_{2}\in\mathbb{R}^{p}$, then $f:\mathbb{R}^{p}\to\mathbb{R}$ is $\theta$-Lipschitz continuous.
\item It there exists a constant $\mu>0$ such that $(x_1-x_2)^{\top}(g(x_1)-g(x_2))\geqslant \mu\Vert x_1-x_2\Vert^{2}$, for any $x_{1}, x_{2}\in\mathbb{R}^{p}$, then $g:\mathbb{R}^{p}\to\mathbb{R}^{p}$ is $\mu$-strongly monotone.
	\end{itemize}\end{definition}
\begin{definition}\cite{Bauschke2011}
 For a closed convex set $S \subseteq \mathbb{R}^{p}$, $P_S(\cdot): \mathbb{R}^{p} \rightarrow S$ is the projection on $S$. $\mathbf{T}_S(\cdot): S \rightrightarrows\mathbb{R}^{p}: x \mapsto$ $\operatorname{cl}\left(\bigcup_{\delta>0} \frac{1}{\delta}(S-x)\right)$ is the tangent cone of $S$. The projection on the tangent cone of $S$ at $x$ is defined as $\Pi_S\left[x,v\right]:=P_{\mathbf{T}_S(x)}(v)=\lim _{\delta \rightarrow 0^{+}} \frac{P_S(x+\delta v)-x}{\delta}$.
\end{definition}
\begin{lem}\cite{Cherukuri2016}\label{tou22}
	Let $S \subseteq \mathbb{R}^{p}$ be a nonempty closed convex set. For any $x_1, x_2 \in S$ and $\xi \in \mathbb{R}^{p}$, it gives that $\left(x_1-x_2\right)^{\top} \Pi_S(x_1, \xi) \leqslant\left(x_1-x_2\right)^{\top} \xi$.
\end{lem}
\subsection{Online Noncooperative Game}
Consider the following constrained online noncooperative game $\Gamma\triangleq\left\{\mathcal{V},J_{i},\Omega_{i}\cap U(t)\right\}$ with $N$ players interacting over an undirected connected graph $\mathcal{G}=\left\{\mathcal{V},\mathcal{E}\right\}$. For each $i\in\mathcal{V}$, the $i$th player's action profile is defined as $x_{i}\in\mathbb{R}^{p_i}$ and $J_{i}(t,x):\mathbb{R}\times\mathbb{R}^{p}\rightarrow\mathbb{R}$ denotes the cost function of the $i$th player, where $p=\sum_{i=1}^Np_i$ and $x=\text{col}\left(x_{1},x_{2},\cdots,x_{N}\right)$ denotes the column vector of all players' action profile. Also, the action profile can be usually rewritten as $x=\col(x_i,x_{-i})$, where $x_{-i}\triangleq\text{col}\left(x_{1},x_{2},\cdots,x_{i-1},x_{i+1},\cdots,x_{N}\right)$. Owing to the regulation originated from reality, the action profiles of players are often restricted by each other. Thus, a coupling nonlinear inequality constraint denoted by the set $U(t)=\left\{x\in\mathbb{R}^{p}\mid\sum_{i=1}^{N}g_{i}(t,x_{i})\leqslant\mathbf{0}_{q},g_{i}(t,\cdot):\mathbb{R}^{p_i}\rightarrow\mathbb{R}^{q}\right\}$ and private closed convex set constraints $\Omega_{i}$ are considered in the constrained online game $\Gamma$.
\par
Specifically, the $i$th player's ultimate goal is to minimize its own cost function over a period of time $[0, T ]$ as follows
\begin{equation}\label{pb}
\begin{aligned}
	\min\limits_{x_{i}\in\Omega_i}&\int_{0}^{T}J_{i}(t,x_{i},x_{-i})\text{d}t\\
	\text{s.t. }&\left(x_{i},x_{-i}\right)\in \Omega\cap U(t),
\end{aligned}
\end{equation}
where $\Omega=\Omega_1\times\Omega_2\times\cdots\times\Omega_N$.
\begin{rem}
The considered time-varying cost functions and coupling constraints arise naturally in many real-world applications, such as demand response in commercial buildings \cite{lesage2020dynamic}, dynamic sparse recovery \cite{dixit2020online}, estimation in sensor networks \cite{hosseini2013online}, target tracking in 2D planes \cite{shahrampour2017distributed}, and Nash-Cournot markets \cite{Lu2020}. For instance, in an online Nash-Cournot market, each firm determines its production quantity to minimize the cumulative production cost while maximizing profit. The cost function is time-varying, as the market demand price and marginal production cost fluctuate with time due to dynamic economic factors. Meanwhile, the market capacity constraint, which limits the total production of all firms, also evolves over time, reflecting changes in resource availability or market saturation. It is worth noting that problem \eqref{pb} formulates an online game, where each player optimizes its own time-varying cost function. The integral form of the cost functions represents the cumulative cost accumulated over time, capturing the continuously changing environment, such as fluctuating prices. It is natural to interpret this formulation within the framework of online learning, where each player sequentially updates its strategy based on the observed information, aiming to minimize the cumulative cost in hindsight, due to the difficulty in directly solving the time-varying problem. This formulation is closely related to the literature on online optimization and online game \cite{li2023survey,yu2024distributed}.
\end{rem}

In the context of the online noncooperative game, the offline best response to problem \eqref{pb} known as the GNE is considered. The definition is shown as follows.

\begin{definition}
	An action profile $x^{*}=\text{col}\left(x_{i}^{*},x_{-i}^{*}\right)\in\Omega\cap U(t)$ is a GNE of the online game \eqref{pb}, if for any $t\in[0,T]$, $i\in\mathcal{V}$ and $x_i:\left(x_{i},x_{-i}^{*}\right)\in\Omega\cap U(t)$, there is
	\begin{equation*}
		\int_{0}^{T}J_{i}(t,x_{i}^{*},x_{-i}^{*})\text{d}t\leqslant\int_{0}^{T}J_{i}(t,x_{i},x_{-i}^{*})\text{d}t.
	\end{equation*}
\end{definition}
\begin{rem}
Note that determining all GNEs is challenging, as each player's decision-making process is influenced by its competitors. One subclass of GNEs, referred to as VGNEs \cite{facchinei2007generalized}, has been extensively investigated due to their favorable economic interpretation in the absence of price discrimination and their better stability properties. Another motivation for studying VGNEs lies in their close connection with variational inequality problems. Specifically, the VGNE $x^{*}$ of the online game \eqref{pb} corresponds to the solution $x^{*} \in \Omega \cap U(t)$ of the following variational inequality
$$
\left\langle \int_{0}^{T}F\left(t,x^{*}\right)\text{d}t,\left(x-x^{*}\right)\right\rangle \geqslant 0, \forall x \in \Omega \cap U(t),
$$
where $F=\text{col}(\partial_{x_1} J_{1}(t,x_{1},x_{-1}),\cdots,\partial_{x_N} J_{N}(t,x_{N},x_{-N}))$. According to \cite{Lu2020}, $x^*$ satisfying the above condition is indeed a VGNE of the online game \eqref{pb}.
\end{rem}
\begin{assum}\label{jiashe11}
The cost function $J_{i}(t,x_{i},x_{-i})$ and the constraint function $g_{i}(t,x_{i})$ in \eqref{pb} are integrable with respect to $t\in\left[0,T\right]$, convex with respect to $x_i$ and Lipschitz continuous over the set $\Omega_{i}$. The pseudo-gradient $F(t,x)$ is $l$-Lipschitz continuous and $m$-strongly monotone with respect to $x$. 
\end{assum}
\begin{assum}\label{jiashe22}
The set of feasible action profile $\Omega^\flat=\{x~|~x\in\Omega, \sum_{i=1}^{N}g_{i}(t,x_{i})\leqslant\mathbf{0}_{q}, t\in[0,T]\}$ is non-empty.
\end{assum}
\begin{rem}\label{youjie}
Assumption \ref{jiashe11} ensures the existence of a VGNE for the online game \eqref{pb}, which is reasonable since practical outputs, such as voltage, typically operate within a certain range. The strong monotonicity of the pseudo-gradient $F$ of is a general condition for most of the NE seeking algorithms for static games \cite{shi2023distributed,ma2022fully,ye2017distributed,gadjov2019distributed,Yang2024,Zeng2019}, which guarantees the algorithms' effectiveness. Also, based on this condition, we can get for any $x$ and $y$, there is $(x-y)^{\top}(F(t,x)-F(t,y))=\sum_{i=1}^N(x_i-y_i)^\top(\partial_{x_i} J_{i}(t,x_{i},x_{-i})-\partial_{y_i} J_{i}(t,y_{i},y_{-i}))\geqslant m \sum_{i=1}^N\|x_i-y_i\|^2.$ Thus, by letting $x_j=y_j$ for $j\neq i$, there is $\left(x_i-y_i\right)^\top\left(\partial_{x_i} J_{i}(t,x_{i},x_{-i})-\partial_{y_i} J_{i}(t,y_{i},y_{-i})\right)\geqslant m \left\|x_i-y_i\right\|^2,$ which implies that $J_i(t, x_i, x_{-i})$ is $m_i$-strongly convex with respect to $x_i$ for $m_i\geqslant m$. In addition, based on Assumption \ref{jiashe11}, the boundedness of $J_i$ and $g_i$ can be inferred, i.e., there exist positive constants $K_{f}$ and $K_{g}$, such that $\Vert J_{i}(t,x_{i},x_{-i})\Vert\leqslant K_{f}$ and $\Vert g_{i}(t,x_{i})\Vert\leqslant K_{g}$ hold. Although Assumption \ref{jiashe11} requires the cost functions to be integrable over any finite interval $[0,T]$, this does not necessarily guarantee that the total accumulated cost remains finite as $T \to \infty$. In fact, the main objective of this paper is to investigate the time-averaged performance of the algorithms over a finite-time horizon, and thus, $T$ is typically confined within a finite range. Nevertheless, since the cost functions are expressed in an integral form, it is reasonable to ensure the rationality of the cost functions. Therefore, one may impose a stronger condition $\int_0^{\infty}\|J_i(t, x_i, x_{-i})\| \text{d} t<\infty$.
\end{rem}
\begin{rem}\label{juhe}
As a special kind of noncooperative game, the online aggregative game can also be similarly defined. Consider an aggregative function $\sigma(x)=\frac{1}{N}\sum_{i=1}^{N}\vartheta_{i}(x_{i})$ constituted by linear mappings $\vartheta_{i}(x_{i}):\mathbb{R}^{p_i}\rightarrow\mathbb{R}^n$, which signify the $i$th player's local contribution to the aggregative function. Then, the $i$th player's cost function can be reformulated as $J_{i}(t,x_{i},x_{-i})=f_{i}(t,x_{i},\sigma(x))$ with a known function $f_{i}:\mathbb{R}\times\mathbb{R}^{p_i}\times\mathbb{R}^{n}\rightarrow\mathbb{R}$.
\end{rem}
\subsection{Performance Metrics}\label{refit}
Due to the time-varying nature of the cost functions and constraints, computing the instantaneous time-varying VGNE may be time-consuming. Thus, two performance metrics, regret $\mathcal{R}$ and fit $\mathcal{F}$, are established to assess the performance of the VGNE seeking algorithms for the online game \eqref{pb}. The regret is given as
\begin{equation*} \mathcal{R}^{\top}=\int_{0}^{T}\sum_{i=1}^{N}\left(J_{i}(t,x_{i}(t),x_{-i}^{*})-J_{i}(t,x_{i}^{*},x_{-i}^{*})\right)\text{d}t,
\end{equation*}
where $x^*$ is the VGNE of online game \eqref{pb}. This integral-type regret serves as a continuous-time analogue of the classical cumulative regret widely used in online game and optimization \cite{li2023survey}, where the discrete summation is replaced by time integration. It measures the cumulative cost gap between the instantaneous players' actions and the VGNE. Similar definitions can be found in \cite{yu2024distributed}. Following the standard notion in \cite{li2023survey}, the proposed algorithm is deemed ``good'' if $\mathcal{R}^\top$ grows sublinearly with respect to $T$, i.e., $\frac{\mathcal{R}^\top}{T} \to 0$ as $T \to \infty$.

Then, the fit measuring the cumulative constraint violation is defined as
\begin{equation*} \mathcal{F}^{\top}=\left\|\left[\int_{0}^{T}\sum_{i=1}^{N}g_{i}(t,x_{i})\text{d}t\right]_{+}\right\|.
\end{equation*}
This integral-type fit also serves as a continuous-time analogue of the discrete-type fit commonly used in online learning \cite{li2023survey}, measuring cumulative constraint violation. It evaluates how closely the output trajectories $x(t)$ by the VGNE seeking algorithms comply with the imposed constraints over time. Similar definitions can also be found in \cite{yu2024distributed}. This definition allows for temporary violations to be offset by subsequent feasible actions, which is particularly meaningful for accumulated quantities such as average power \cite{paternain2016online}. Similar to the regret, an algorithm is considered effective if the fit grows sublinearly with respect to $T$.

Define $g_{i, j}(t, \cdot): \mathbb{R}^{p_i} \rightarrow \mathbb{R}$ as the $j$th component of $g_i(t, \cdot)$, i.e., $g_i(t, \cdot)=\operatorname{col}\left(g_{i, 1}(t, \cdot), \cdots, g_{i, q}(t, \cdot)\right)$. Then, define $\mathcal{F}_j^\top:=\int_0^T \sum_{i=1}^N g_{i, j}\left(t, x_i\right) \text{d} t$ as the $j$th component of the constraint integral. It gives that $\mathcal{F}^\top=\sqrt{\sum_{j=1}^q\left[\mathcal{F}_j^\top\right]_{+}^2}.$
\section{Main Results}\label{MR}
In this section, two continuous-time distributed VGNE seeking algorithms are first developed for the online noncooperative and aggregative games and their performance analyses are provided. Then, an event-triggered VGNE seeking algorithm is proposed for the online noncooperative game to reduce the communication load. Finally, the effect of measurement noise on the algorithm's performance is investigated.
\subsection{Continuous-Time VGNE Seeking Algorithm}
In contrast to general distributed optimization problems, in a noncooperative game, each player must acquire information about the actions of all other players to evaluate its own cost. This poses a challenge in distributed scenarios. To address this issue, a passivity-based estimation method, inspired by \cite{romano2019dynamic}, is employed. Let $\Upsilon^{i}_{j}$ denotes $i$th player's estimation on $j$th player's action $x_{j}$. Denote $\Upsilon^{i}=\text{col}\left(\Upsilon^{i}_{1},\cdots,\Upsilon^{i}_{N}\right)$ and $\Upsilon^{i}_{-i}=\text{col}\left(\Upsilon^{i}_{1}, \cdots, \Upsilon^{i}_{i-1},\Upsilon^{i}_{i+1},\cdots,\Upsilon^{i}_{N}\right)$. Clearly, it can be deduced that $\Upsilon^{i}_i=x_i$.
\begin{corollary}\label{lip_extended}\cite{Bianchi2021}
The extended pseudo-gradient $\mathbf{F}:=\text{col}(\partial_{x_{1}}J_{1}(t,x_{1},\Upsilon^{1}_{-1}),\cdots,\partial_{x_{N}}J_{N}(t,x_{N},\Upsilon^{N}_{-N}))$ is $l_{0}$-Lipschitz continuous, for some $l_{0}\in[m,l]$.
\end{corollary}

A novel continuous-time distributed VGNE seeking algorithm for constrained online game \eqref{pb} is proposed as follows
\begin{equation}\label{algo1}
	\left\{
	\begin{aligned}
		\dot{x}_{i}  = &\Pi_{\Omega_{i}}[x_{i},-\partial_{x_{i}}J_{i}(t,\Upsilon^{i})-\mu_{i}^{\top}\partial_{x_i} g_{i}(t,x_{i})\\
&-kR_{i}\sum_{j\in\mathcal{N}_{i}}(\Upsilon^{i}-\Upsilon^{j})],\\
		\dot{\Upsilon}_{-i}^{i}  =& -kS_{i}\sum_{j\in\mathcal{N}_{i}}(\Upsilon^{i}-\Upsilon^{j}),\\
		\dot{\mu}_{i}  =& \Pi_{\mathbb{R}_{\geqslant0}^{q}}[\mu_{i},g_{i}(t,x_{i})-K_{\mu}\sum_{j\in\mathcal{N}_{i}}\text{sgn}(\mu_{i}-\mu_{j})],
	\end{aligned}
	\right.
	\end{equation}
	where $\partial_{x_i} g_{i}(t,x_{i})=\frac{\partial g_i(t,x_i)}{\partial x_i}$ represents the Jacobian of the local constraint function $g_{i}(t,x_{i})$ with respect to $x_i$, $k,K_{\mu}>0$ are constants to be defined later, $$
	R_{i}=\begin{bmatrix}
		\mathbf{0}_{p_i\times n_{<i}} & \mathbf{I}_{p_i}&\mathbf{0}_{p_i\times n_{>i}}
	\end{bmatrix}, $$$$
	S_{i}=\begin{bmatrix}
		\mathbf{I}_{n_{<i}} & \mathbf{0}_{n_{<i}\times p_i}&\mathbf{0}_{n_{<i}\times n_{>i}}\\
		\mathbf{0}_{n_{>i}\times n_{<i}}&\mathbf{0}_{n_{>i}\times p_i}&\mathbf{I}_{n_{>i}}
	\end{bmatrix},
	$$
with $n_{<i}=\sum_{j<i, i, j \in \mathcal{V}} p_j$ and $n_{>i}=\sum_{j>i, i, j \in \mathcal{V}} p_j$.

In a static game, the KKT conditions coincide with the static VGNE conditions, and a pseudo-gradient-based algorithm can seek the VGNE. However, in an online game with time-varying costs, exact tracking of the instantaneous VGNE is generally more complex; instead, sublinear regret ensures that the time-average performance of the algorithm asymptotically matches that of the offline static VGNE, i.e., $\frac{\mathcal{R}^\top}{T}\rightarrow0$, indicating that the online algorithm is ``good'' on average. 

Also, in a time-varying game that aims to track the instantaneous time-varying VGNE, one typically needs access to temporally dependent terms (e.g., $\partial_t J_i(t,x_i,x_{-i})$ and $\partial_t g_i(t,x_i)$) or to design additional measurement/estimation mechanisms. Such schemes are generally computationally intensive and often rely on global information or stronger assumptions. In contrast, the online-regret viewpoint adopted here eliminates the need for $\partial_t$ terms and Hessian matrices, thereby achieving a distributed implementation with lower computational complexity by guaranteeing sublinear regret (hence vanishing time-average performance gap).

	\begin{rem}\label{juhe_algo}
	As is discussed in Remark \ref{juhe}, when dealing with the online aggregative game, a distributed VGNE seeking algorithm similar to \eqref{algo1} can be formulated as
	\begin{equation}\label{algo_juhe}
		\left\{
		\begin{aligned}
			\dot{x}_{i}  = &\Pi_{\Omega_{i}}[x_{i},-\Phi_{i}(t,x_{i},\phi_{i})-\mu_{i}^{\top}\partial_{x_i} g_{i}(t,x_{i})],\\
			\dot{\mu}_{i}=& \Pi_{\mathbb{R}_{\geqslant0}^{q}}[\mu_{i},g_{i}(t,x_{i})-K_{\mu}\sum_{j\in\mathcal{N}_{i}}\text{sgn}(\mu_{i}-\mu_{j})],\\
			\dot{\psi}_{i}=&-\omega[\sum_{j\in\mathcal{N}_{i}}\text{sgn}(\phi_{i}-\phi_{j})],\\
			\phi_{i}=&\psi_{i}+\vartheta_{i}(x_{i}),
		\end{aligned}
		\right.
	\end{equation}
	where $\omega>0$ is a constant, the $i$th player's local estimation on the gradient information $\partial_{x_{i}}f_{i}(t,x_{i},\sigma(x))$ is denoted by $$\Phi_{i}(t,x_{i},\phi_{i})\triangleq(\partial_{x_i}f_i(t,\cdot,\sigma)+\frac{1}{N}\partial_\sigma f_i^\top(t,x_i,\cdot)\partial\vartheta_i)|_{\sigma=\phi_i},$$ and it is apparent that $\Phi_{i}(t,x_{i},\phi_{i})=\partial_{x_{i}}f_{i}(t,x_{i},\sigma(x))$ when $\phi_{i}=\sigma(x)$. Besides, all the other symbols in algorithm \eqref{algo_juhe} are the same as those in algorithm \eqref{algo1}. Let $\sup_{t\in[0,+\infty)}\Vert\dot{\vartheta}_i(x_i)\Vert\leqslant\iota$. According to Theorem 2 in \cite{Chen2012}, when $\omega>(N-1)\iota$, there exists a time bound $T_{0}$ such that $\lim_{t\rightarrow T_0}\Vert \phi_i(t)-\sigma(x(t))\Vert=0$ for any $i\in\mathcal{V}$. 
	\end{rem}
\begin{thm}\label{dingli11}
Suppose that Assumptions \ref{jiashe11} and \ref{jiashe22} hold. For any $T\geqslant0$, when $x(0)\in\Omega$, $\mu(0)=\mathbf{0}_{Nq}$, $m+\frac{m_{0}}{2}>l$, $k>\frac{(l-l_{0})^{2}+(4m+2m_{0})l_{0}}{(4m+2m_{0}-4l)\lambda_{2}(\mathbf{L})}$ and $K_\mu\geqslant NK_g$, based on the VGNE seeking algorithm \eqref{algo1}, the following regret and fit bounds hold
\begin{equation*}
	\left\{\begin{aligned}
		\mathcal{R}^{\top}&\leqslant \frac{1}{2}\Vert\Upsilon(0)-\Upsilon^{*}\Vert^{2},\\
		\mathcal{F}^{\top}&\leqslant 2N\sqrt{K_{f}T}+\sqrt{N}\Vert \Upsilon(0)-\Upsilon^{*}\Vert,
	\end{aligned}
	\right.
\end{equation*}
in which $m_0=\min_{i\in\mathcal{V}}\{m_i\}$, $\Upsilon=\text{col}(\Upsilon^{1},\cdots,\Upsilon^{N}), \mu=\text{col}(\mu_{1},\cdots,\mu_{N}), \Upsilon^{*}=\mathbf{1}_{N}\otimes x^{*},$ $x^*$ is the VGNE of online game \eqref{pb}, $\mu^{*}=\mathbf{1}_{N}\otimes\bar{\mu}^{*}$, $\bar{\mu}^{*}\in\mathbb{R}_{\geqslant0}^{q}$ is a nonnegative vector, which can be arbitrarily chosen and $\mathbf{L}=\mathcal{L}\otimes \mathbf{I}_{Np}$. In other words, we have $\mathcal{R}^{\top}=\mathcal{O}(1)$ and $\mathcal{F}^{\top}=\mathcal{O}(\sqrt{T})$.
\end{thm}
\begin{proof}
Select the candidate Lyapunov function as follows $$V(t)=\frac{1}{2}\left\|\Upsilon-\Upsilon^{*}\right\|^{2}+\frac{1}{2}\left\|\mu-\mu^*\right\|^{2}.$$
\par
Taking the time derivative of $V$ along the direction of \eqref{algo1} and based on Lemma \ref{tou22}, it yields that
\begin{equation*}\label{dv}
\begin{aligned}
\dot{V}=&(x-x^{*})^{\top}\dot{x}+\sum_{i=1}^N(\Upsilon^i_{-i}-\Upsilon_{-i}^{i*})^{\top}\dot{\Upsilon}^i_{-i}+(\mu-\mu^{*})^{\top}\dot{\mu}\\
\leqslant&\sum_{i=1}^{N}(x_{i}-x_{i}^{*})^{\top}\left[-\partial_{x_{i}}J_{i}(t,\Upsilon^{i})-kR_{i}\sum_{j=1}^{N}a_{ij}(\Upsilon^{i}-\Upsilon^{j})\right]\\
&-\sum_{i=1}^{N}(x_{i}-x_{i}^{*})^{\top}\mu_{i}^{\top}\partial_{x_i} g_{i}(t,x_{i})\\
&+\sum_{i=1}^{N}(\mu_{i}-\mu_{i}^{*})^{\top}\left[g_{i}(t,x_{i})-K_{\mu}\sum_{j=1}^{N}a_{ij}\text{sgn}(\mu_{i}-\mu_{j})\right]\\
&+\sum_{i=1}^{N}(\Upsilon_{-i}^{i}-\Upsilon_{-i}^{i^{*}})^{\top}\left[-kS_{i}\sum_{j=1}^{N}a_{ij}(\Upsilon^{i}-\Upsilon^{j})\right].
\end{aligned}
\end{equation*}
\par
To facilitate the subsequent analysis, deploy $V_{1}$ and $V_{2}$ to represent the terms in the above formula with
\begin{equation*}
\begin{aligned}
	V_{1}=&\sum_{i=1}^{N}(x_{i}^{*}-x_{i})^{\top}\left[\partial_{x_{i}}J_{i}(t,\Upsilon^{i})+kR_{i}\sum_{j=1}^{N}a_{ij}(\Upsilon^{i}-\Upsilon^{j})\right]\\
	&+\sum_{i=1}^{N}(\Upsilon_{-i}^{i}-\Upsilon_{-i}^{i^{*}})\left[-kS_{i}\sum_{j=1}^{N}a_{ij}(\Upsilon^{i}-\Upsilon^{j})\right],
\end{aligned}
\end{equation*}
and
\begin{equation*}
	\begin{aligned}
	V_{2}=&\sum_{i=1}^{N}(x_{i}^{*}-x_{i})^{\top}\mu_{i}^{\top}\partial_{x_i} g_{i}(t,x_{i})\\
	&+\sum_{i=1}^{N}(\mu_{i}-\mu_{i}^{*})^{\top}\left[g_{i}(t,x_{i})-K_{\mu}\sum_{j=1}^{N}a_{ij}\text{sgn}(\mu_{i}-\mu_{j})\right].
	\end{aligned}
\end{equation*}
As for $V_{1}$, it can be reformulated as
\begin{equation}\label{dv1}
	\begin{aligned}
		V_{1}=&\sum_{i=1}^{N}(x_{i}^{*}-x_{i})^{\top}\partial_{x_{i}}J_{i}(t,\Upsilon^{i})-k(\Upsilon-\Upsilon^{*})^{\top}\mathbf{L}(\Upsilon-\Upsilon^{*})\\
		=&\sum_{i=1}^{N}(x_{i}^{*}-x_{i})^{\top}\left(\partial_{x_{i}}J_{i}(t,\Upsilon^{i})-\partial_{x_{i}}J_{i}(t,\Upsilon^{i^{*}})\right)\\
		&+\sum_{i=1}^{N}(x_{i}^{*}-x_{i})^{\top}\left(\partial_{x_{i}}J_{i}(t,x_{i}^{*},x_{-i}^{*})-\partial_{x_{i}}J_{i}(t,x_{i},x_{-i}^{*})\right)\\
		&+\sum_{i=1}^{N}(x_{i}^{*}-x_{i})^{\top}\partial_{x_{i}}J_{i}(t,x_{i},x_{-i}^{*})\\
		&-k(\Upsilon-\Upsilon^{*})^{\top}\mathbf{L}(\Upsilon-\Upsilon^{*}).
	\end{aligned}
\end{equation}
On account of Assumption \ref{jiashe11}, it can be further scaled that 
\begin{equation}\label{dv11}
	\begin{aligned}
		&\sum_{i=1}^{N}(x_{i}^{*}-x_{i})^{\top}\left(\partial_{x_{i}}J_{i}(t,\Upsilon^{i})-\partial_{x_{i}}J_{i}(t,\Upsilon^{i^{*}})\right)\\
		\leqslant&\frac{l_{0}+l}{\sqrt{N}}\Vert \Upsilon^{\parallel}-\Upsilon^{*}\Vert\Vert\Upsilon^{\perp}\Vert-\frac{m}{N}\Vert\Upsilon^{\parallel}-\Upsilon^{*}\Vert^{2}+l_{0}\Vert\Upsilon^{\perp}\Vert^{2},
	\end{aligned}
\end{equation}
where $\Upsilon=\Upsilon^{\|}+\Upsilon^{\perp}$ with $\Upsilon^{\perp} \in \Psi^{\perp}$ and $\Upsilon^{\|} \in \Psi^{\|}$ with $\Psi^{\|}=\left\{\mathbf{1}_N \otimes \gamma \mid \gamma \in \mathbb{R}^{Np}\right\}$ and $\Psi^{\perp}$ is the orthogonal complement space of $\Psi^{\|}$. The detailed process of \eqref{dv11} is similar to the proof of Lemma 4 in \cite{Bianchi2021}, and thus it is omitted herein. Then, owing to the Lipschitz continuity of $\partial_{x_{i}}J_{i}(t,x_{i},x_{-i})$, it can be further deduced that
\begin{equation}\label{v1_lip}
	\begin{aligned}
	&\sum_{i=1}^{N}(x_{i}^{*}-x_{i})^{\top}\left(\partial_{x_{i}}J_{i}(t,x_{i}^{*},x_{-i}^{*})-\partial_{x_{i}}J_{i}(t,x_{i},x_{-i}^{*})\right)\\
	\leqslant&l\Vert x-x^{*}\Vert^{2}.
	\end{aligned}
\end{equation}
What's more, by referring to the strong convexity of $\partial_{x_{i}}J_{i}(t,x_{i},x_{-i})$ with respect to $x_i$, it can be reached that 
\begin{equation}\label{strong_convex}
	\begin{aligned}
	&\sum_{i=1}^{N}(x_{i}^{*}-x_{i})^{\top}\partial_{x_{i}}J_{i}(t,x_{i},x_{-i}^{*})\\
	\leqslant&\sum_{i=1}^{N}[J_{i}(t,x_{i}^{*},x_{-i}^{*})-J_{i}(t,x_{i},x_{-i}^{*})]-\frac{m_{0}}{2}\Vert x-x^{*}\Vert^{2}.
	\end{aligned}
\end{equation}
Meanwhile, according to Lemma \ref{yinli11}, it yields that $\mathbf{L}\Upsilon^{\parallel}=\mathbf{L}\Upsilon^{\parallel^{*}}=\mathbf{0}_{N^2p}$. Thus, the following inequality holds
	\begin{equation}\label{tzz}
-k(\Upsilon-\Upsilon^{*})^{\top}\mathbf{L}(\Upsilon-\Upsilon^{*})\leqslant-k\lambda_{2}(\mathbf{L})\Vert\Upsilon^{\perp}\Vert^{2}.
\end{equation}
Taking \eqref{dv11}-\eqref{tzz} back into \eqref{dv1}, it can be derived that
\begin{equation*}
	V_{1}\leqslant-\bar{\mathcal{R}}^{\top}-\lambda_{\min}(\mathbf{M})\Vert\Upsilon-\Upsilon^{*}\Vert^{2},
\end{equation*}
where $\bar{\mathcal{R}}^{\top}=\sum_{i=1}^{N}\left[J_{i}(t,x_{i},x_{-i}^{*})-J_{i}(t,x_{i}^*,x_{-i}^{*})\right]$ and the matrix $\mathbf{M}$ is defined as
\begin{equation*}
	\mathbf{M}=\begin{bmatrix}
		\frac{2m+m_{0}-2l}{2N}&-\frac{l_{0}+l}{2\sqrt{N}}\\
		-\frac{l_{0}+l}{2\sqrt{N}}&k\lambda_{2}(\mathbf{L})-l_{0}
	\end{bmatrix},
\end{equation*}
which is positive definite based on the parameter setting.

As for $V_{2}$, based on the definition of $\mu^*$, it yields that
	\begin{align}\label{dv2}
		V_{2}\leqslant&\sum_{i=1}^{N}\left(g_{i}(t,x_{i}^{*})-g_{i}(t,x_{i})\right)^{\top}\mu_{i}+\sum_{i=1}^{N}(\mu_{i}-\mu_{i}^{*})^{\top}g_{i}(t,x_{i})\nonumber\\
		&-\frac{K_{\mu}}{2}\sum_{i=1}^{N}\sum_{j=1}^{N}a_{ij}\Vert\mu_{i}-\mu_{j}\Vert_{1}\nonumber\\
		=&\sum_{i=1}^{N}g_{i}(t,x_{i}^{*})^{\top}\mu_{i}-\sum_{i=1}^{N}g_{i}(t,x_{i})^{\top}\mu_{i}^{*}\\
		&-\frac{K_{\mu}}{2}\sum_{i=1}^{N}\sum_{j=1}^{N}a_{ij}\Vert\mu_{i}-\mu_{j}\Vert_{1}\nonumber.
	\end{align}

	\textbf{Step 1}: By letting $\bar{\mu}^{*}=\mathbf{0}_{q}$ and $K_\mu\geqslant NK_g$, based on the Appendix B in \cite{yu2024distributed}, it gives that $V_{2}\leqslant 0$.
\par
Gathering the analysis for $V_{1}$ and $V_{2}$, it holds that
\begin{equation}\label{regret}
	\bar{\mathcal{R}}^{\top}\leqslant-\dot{V}-\lambda_{\min}(\mathbf{M})\Vert\Upsilon-\Upsilon^{*}\Vert^{2}
	\leqslant-\dot{V}.
\end{equation}
By integrating time $t$ from $0$ to $T$ on both sides of \eqref{regret}, it yields that
\begin{equation*}
	\mathcal{R}^{\top}\leqslant-\int_{0}^{T}\dot{V}\text{d}t
	\leqslant\frac{1}{2}\left\|\Upsilon(0)-\Upsilon^{*}\right\|^{2}.
\end{equation*}
Thus it can be concluded that $\mathcal{R}^{\top}=\mathcal{O}(1)$.
\par
\textbf{Step 2}: By letting \begin{equation*}\label{gamma}
		\bar{\mu}^*_i=
		\begin{cases}
			0,&\text{if }\mathcal{F}_{i}^{\top}\leqslant0,\\
			\frac{\mathcal{F}_{i}^{\top}}{N},&\text{if }\mathcal{F}_{i}^{\top}>0,
		\end{cases}
	\end{equation*}where $i\in\left\{1,2,\cdots,q\right\}$ and $\bar{\mu}^*=\col(\bar{\mu}^*_1,\bar{\mu}^*_2,\cdots,\bar{\mu}^*_q)$, \eqref{dv2} can also be scaled as
$
	V_{2}\leqslant-\bar{\mu}^{*\top}\sum_{i=1}^{N}g_{i}(t,x_{i})
$. Therefore, $\dot{V}$ can be further scaled as
\begin{equation*}\label{fitbound}
	\dot{V}\leqslant-\bar{\mathcal{R}}^{\top}-\bar{\mu}^{*\top}\sum_{i=1}^{N}g_{i}(t,x_{i}).
	\end{equation*}Based on the boundedness of $J_{i}$, it can be derived that
	\begin{equation*}
		\frac{1}{N}\Vert\mathcal{F}\Vert^{2}\leqslant 2NK_{f}T+\frac{1}{2}\left\|\Upsilon(0)-\Upsilon^{*}\right\|^{2}+\frac{1}{2N}\Vert\mathcal{F}\Vert^{2}.
	\end{equation*}
Finally, through transposition and scaling, it gives that
\begin{equation*}
	\mathcal{F}^{\top}\leqslant 2N\sqrt{K_{f}T}+\sqrt{N}\Vert\Upsilon(0)-\Upsilon^{*}\Vert.
\end{equation*}
Thus, $\mathcal{F}^{\top}=\mathcal{O}(\sqrt{T})$ can be reached.
\end{proof}
\begin{corollary}\label{juhe_v}
Suppose that Assumptions \ref{jiashe11} and \ref{jiashe22} hold. For any $T> T_0$, when $x(0)\in\Omega$, $\omega>(N-1)\iota$, $\mu(0)=\mathbf{0}_{Nq}$, $m+\frac{m_{0}}{2}\geqslant l$ and $K_\mu\geqslant NK_g$, based on the VGNE seeking algorithm \eqref{algo_juhe} for online aggregative game, the following regret and fit bounds hold
\begin{equation*}
	\left\{\begin{aligned}
		\mathcal{R}^{\top}&\leqslant \frac{1}{2}\Vert x(0)-x^{*}\Vert^{2},\\
		\mathcal{F}^{\top}&\leqslant 2N\sqrt{K_{f}T}+\sqrt{N}\Vert x(0)-x^{*}\Vert,
	\end{aligned}
	\right.
\end{equation*}
where $x^{*}$ is the VGNE of the online aggregative game. In other words, we have $\mathcal{R}^{\top}=\mathcal{O}(1)$ and $\mathcal{F}^{\top}=\mathcal{O}(\sqrt{T})$.
\end{corollary}
\begin{proof}
	According to the analysis in Remark \ref{juhe_algo}, when $\omega>(N-1)\iota$. there exists a finite time bound $T_{0}$ such that $\lim_{t\rightarrow T_0}\Vert \phi_{i}(t)-\sigma(x(t))\Vert =0$ holds for any $i\in\mathcal{V}$. Thus, when $t\geqslant T_{0}$, the average tracking technique employed in \eqref{algo_juhe} realizes the accurate estimation on the aggregative function $\sigma(x)$ and it holds that $\Phi_{i}(t,x_{i},\phi_{i})=\partial_{x_{i}}f_{i}(t,x_{i},\sigma(x))$. That is to say, when $t\geqslant T_{0}$, the VGNE seeking algorithm \eqref{algo_juhe} can be transformed as 
		\begin{equation}\label{algo_juhe2}
		\left\{
		\begin{aligned}
			\dot{x}_{i}  = &\Pi_{\Omega_{i}}[x_{i},-\partial_{x_{i}}f_{i}(t,x_{i},\sigma(x))-\mu_{i}^{\top}\partial_{x_i} g_{i}(t,x_{i})],\\
			\dot{\mu}_{i}=& \Pi_{\mathbb{R}_{\geqslant0}^{q}}[\mu_{i},g_{i}(t,x_{i})-K_{\mu}\sum_{j\in\mathcal{N}_{i}}\text{sgn}(\mu_{i}-\mu_{j})].	\end{aligned}
		\right.
	\end{equation}
	Similar to the proof of Theorem \ref{dingli11}, the candidate Lyapunov function is selected as
	\begin{equation*}
		\bar{V}(t)=\frac{1}{2}\Vert x-x^{*}\Vert ^{2} +\frac{1}{2}\Vert \mu-\mu^{*}\Vert^{2}.
	\end{equation*}
	Likewise, by taking the time derivative of $\bar{V}(t)$ along the direction of $\eqref{algo_juhe2}$ and based on Lemma \ref{tou22}, it gives that
	\begin{equation}\label{juhe_vv}
		\begin{aligned}
			\dot{\bar{V}}\leqslant&\sum_{i=1}^{N}(x_{i}-x_{i}^{*})^{\top}(-\partial_{x_{i}}f_{i}(t,x_{i},\sigma(x))-\mu_{i}^{\top}\partial_{x_i} g_{i}(t,x_{i}))\\
			&+\sum_{i=1}^{N}(\mu_{i}-\mu_{i}^{*})^{\top}\left[g_{i}(t,x_{i})-K_{\mu}\sum_{j=1}^{N}a_{ij}\operatorname{sgn}(\mu_{i}-\mu_{j})\right].
		\end{aligned}
	\end{equation}
Based on Assumptions \ref{jiashe11} and similar to \eqref{dv1}-\eqref{strong_convex}, we can get
	\begin{equation}\label{sm}
		\begin{aligned}
&-\sum_{i=1}^{N}(x_{i}-x_{i}^{*})^{\top}\partial_{x_{i}}f_{i}(t,x_{i},\sigma(x))\\
			=&-\sum_{i=1}^{N}(x_{i}-x_{i}^{*})^{\top}\left(\partial _{x_{i}}J_{i}(t,x_{i},x_{-i})-\partial _{x_{i}}J_{i}(t,x_{i}^{*},x_{-i}^{*})\right)\\
			&-\sum_{i=1}^{N}(x_{i}-x_{i}^{*})^{\top}\left(\partial _{x_{i}}J_{i}(t,x_{i}^{*},x_{-i}^{*})-\partial _{x_{i}}J_{i}(t,x_{i},x_{-i}^{*})\right)\\
			&-\sum_{i=1}^{N}(x_{i}-x_{i}^{*})^{\top}\partial _{x_{i}}J_{i}(t,x_{i},x_{-i}^{*})\\
			\leqslant&-m\Vert x-x^{*}\Vert^{2}+l\Vert x-x^{*}\Vert^{2}\\
			&+\sum_{i=1}^{N}[J_{i}(t,x_{i}^{*},x_{-i}^{*})-J_{i}(t,x_{i},x_{-i}^{*})]-\frac{m_{0}}{2}\Vert x-x^{*}\Vert^{2}.
		\end{aligned}
	\end{equation}
	Then, the remaining terms in \eqref{juhe_vv} are scaled in the same manner as those in \eqref{dv2}.

	\textbf{Step 1}: Likewise, by letting $\bar{\mu}^{*}=\mathbf{0}_{q}$, $K_{\mu}\geqslant NK_{g}$ and $m+\frac{m_{0}}{2}\geqslant l$, along with \eqref{sm}, \eqref{juhe_vv} can be further scaled as
	\begin{equation}\label{R2}
		\bar{\mathcal{R}}^{\top}\leqslant-\dot{\bar{V}}-(m+\frac{m_{0}}{2}-l)\Vert x-x^{*}\Vert^{2}\leqslant-\dot{\bar{V}}.
	\end{equation}
	Then, by integrating time $t$ from $0$ to $T$ on both sides of \eqref{R2}, it yields that
	\begin{equation*}
		\mathcal{R}^{\top}\leqslant-\int_{0}^{T}\dot{\bar{V}}dt\leqslant\frac{1}{2}\Vert x(0)-x^{*}\Vert^{2}.
	\end{equation*}
	That is to say, $\mathcal{R}^{\top}=\mathcal{O}(1)$ holds.

\textbf{Step 2}: Also, through a similar process in Theorem \ref{dingli11}, it can be derived that
	\begin{equation*}
		\mathcal{F}^{\top}\leqslant 2N\sqrt{K_{f}T}+\sqrt{N}\Vert x(0)-x^{*}\Vert.
	\end{equation*}
	Thus, $\mathcal{F}^{\top}=\mathcal{O}(\sqrt{T})$ can be reached.
\end{proof}
\begin{rem}\label{ml}
It is noteworthy that the condition $m+\frac{m_{0}}{2}>l$ in Theorem \ref{dingli11} is applicable to a wide range of real-life online games such as the Nash-Cournot game \cite{li2023survey} and UAV swarm game \cite{Wang2023}. Taking the UAV swarm game as an example, the $i$th UAV's cost function is given as $J_{i}(t,x_{i},x_{-i})=a_{i}(t)\Vert x_{i}-r_{i}(t)\Vert^{2}+b_{i}(t)\sum_{j\in\mathcal{N}_{i}}\Vert x_{i}-x_{j}\Vert^{2},$ where $x_i\in\mathbb{R}^2$ is the $i$th UAV's two dimensional coordinates, $a_{i}(t)$ and $b_{i}(t)$ are time-varying weight coefficients to balance the importance of task exploration and maintaining communication, $r_{i}(t)$ is a time-varying target. Then, through direct calculation, it can be obtained that
\begin{equation*}\label{xishu}
	\left\{\begin{aligned}
	m&=\min_{i}\{2a_{i}(t)\},\\
	m_{0}&=\min_{i}\{2(a_{i}(t)+b_{i}(t)\vert \mathcal{N}_{i}\vert)\},\\
	l&=\max_{i}\{2(a_{i}(t)+b_{i}(t)\vert \mathcal{N}_{i}\vert)\}.
	\end{aligned}
	\right.
\end{equation*}
Thus, the condition $m+\frac{m_{0}}{2}>l$ can be expressed as 
\begin{equation}\label{hengchengli}
\begin{aligned}
	&\min_{i}\{2a_{i}(t)\}+\min_{i}\{a_{i}(t)+b_{i}(t)\vert \mathcal{N}_{i}\vert\}\\
	>&\max_{i}\{2(a_{i}(t)+b_{i}(t)\vert \mathcal{N}_{i}\vert)\}.
\end{aligned}
\end{equation}
Normally speaking, the importance of completing the given task is often far greater than the importance of maintaining communication, that is to say, $a_{i}(t)\gg b_{i}(t)\vert \mathcal{N}_{i}\vert$ holds for any $i\in\mathcal{V}$ and $t>0$. Based on the above analysis, \eqref{hengchengli} can be scaled as $ \min_{i}\{3a_{i}(t)\}>\max_{i}\{2(a_{i}(t)+b_{i}(t)\vert \mathcal{N}_{i}\vert)\},$ which is easy to be satisfied. The similar analysis can also be applied to the Nash-Cournot game. Thus, the condition $m+\frac{m_{0}}{2}>l$ is not restrictive in real-life scenarios. Apart from that, the requirement on $m,~m_{0}$ and $l$ is relaxed to $m+\frac{m_{0}}{2}\geqslant l$ for the online aggregative game according to Corollary \ref{juhe_v}.
\end{rem}
\subsection{Event-Triggered VGNE Seeking Algorithm}
To reduce unnecessary communication among players, in this subsection, a distributed VGNE seeking algorithm with dynamic event-triggered mechanism is proposed.
\par
Specifically, for player $i$, $t_{k}^{i}$ stands for its $k$th communication instant while the set $\left\{t^{i}_{1},t^{i}_{2},\cdots,t_{k}^{i},\cdots\right\}$ denotes for its communication instant sequence. Define $\hat{\Upsilon}^{i}(t)=\Upsilon^{i}(t_{k}^{i}), ~\hat{\mu}_{i}(t)=\mu_{i}(t_{k}^{i}),~\text{for}~ t \in[t_{k}^{i},t_{k+1}^{i})$ as the available information of its neighbors at $t$ and $e_{\Upsilon}^{i}(t)=\hat{\Upsilon}^{i}(t)-\Upsilon^{i}(t),~ e_{\mu}^{i}(t)=\hat{\mu}_{i}(t)-\mu_{i}(t)$ as the measurement errors.
\par
Then, a distributed VGNE seeking algorithm with dynamic event-triggered mechanism is put forward as follows
\begin{equation}\label{algo2}
	\left\{
	\begin{aligned}
		\dot{x}_{i}=&\Pi_{\Omega_{i}}[x_{i},-\partial_{x_{i}}J_{i}(t,\Upsilon^{i})-\mu_{i}^{\top}\partial_{x_i} g_{i}(t,x_{i})\\
&-kR_{i}\sum_{j\in\mathcal{N}_{i}}(\hat{\Upsilon}^{i}-\hat{\Upsilon}^{j}),\\
		\dot{\Upsilon}_{-i}^{i}=&-kS_{i}\sum_{j\in\mathcal{N}_{i}}(\hat{\Upsilon}^{i}-\hat{\Upsilon}^{j}),\\
		\dot{\mu}_{i}=&\Pi_{\mathbb{R}_{\geqslant 0}^{q}}[\mu_{i},g_{i}(t,x_{i})-2K_{\mu}\sum_{j\in\mathcal{N}_{i}}\text{sgn}(\hat{\mu}_{i}-\hat{\mu}_{j})],
	\end{aligned}
	\right.
\end{equation}
where all the other symbols have the same meaning as are mentioned in \eqref{algo1}. The dynamic event-triggered mechanism is designed as follows
\begin{equation}\label{triggerc}
	\begin{aligned}
	t_{k+1}^{i}=&\inf_{t>t_{k}^{i}}\{t\mid4d_{i}\Vert e_{\Upsilon}^{i}\Vert ^{2}>\sum_{j=1}^{N}a_{ij}\Vert\Upsilon^{i}(t_{k}^{i})-\Upsilon^{j}(t_{k}^{j})\Vert^{2}+\beta_{i}\\
	&\text{ or }6\sqrt{q}N\Vert e_{\mu}^{i}\Vert>\sum_{j=1}^{N}a_{ij}\Vert\mu_{i}(t_{k}^{i})-\mu_{j}(t_{k}^{j})\Vert_1+\gamma_{i}\},
	\end{aligned}
\end{equation}
where the interval variables $\beta_{i}, \gamma_{i}$ with $\beta_{i}(0)>0, \gamma_{i}(0)>0$ are updated respectively as follows
\begin{equation}\label{update}
	\left\{
	\begin{aligned}
	\dot{\beta}_{i}&=\eta(-\beta_{i}+\sum_{j=1}^{N}a_{ij}\Vert\Upsilon^{i}(t_{k}^{i})-\Upsilon^{j}(t_{k}^{j})\Vert^{2}-4d_{i}\Vert e_{\Upsilon}^{i}\Vert^{2}),\\
	\dot{\gamma}_{i}&=\eta(-\gamma_{i}+\sum_{i=1}^{N}a_{ij}\Vert\mu_{i}(t_{k}^{i})-\mu_{j}(t_{k}^{j})\Vert_{1}-6\sqrt{q}N\Vert e_{\mu}^{i}\Vert),
	\end{aligned}\right.
\end{equation}
in which $\eta$ is a small positive coefficient.
\begin{rem}
As indicated in the review of event-triggered control schemes \cite{ge2021dynamic}, the threshold functions in the event-triggered mechanism \eqref{triggerc} include not only the system information terms $\sum_{j=1}^{N}a_{ij}\Vert\Upsilon^{i}(t_{k}^{i})-\Upsilon^{j}(t_{k}^{j})\Vert^{2}$ and $\sum_{j=1}^{N}a_{ij}\Vert\mu_{i}(t_{k}^{i})-\mu_{j}(t_{k}^{j})\Vert_1$ but also auxiliary variables $\beta_{i}$ and $\gamma_{i}$ governed by the dynamics \eqref{update}. This feature classifies it as a dynamic event-triggered mechanism. In contrast, the static event-triggered mechanism in \cite{yu2024distributed} is defined as \begin{equation*}
					t_{i}^{l+1}:=\inf_{t>t_{i}^{l}}\{t\mid \Vert e_{i}(t)\Vert > \frac{\sum_{j\in\mathcal{N}_{i}}\Vert \hat{\mu}_{i}-\hat{\mu}_{j}\Vert_{1}}{6N\sqrt{q}}+\frac{\sigma e^{-\iota t}}{3N^{2}K_{\mu}\sqrt{q}}\}, 
				\end{equation*} where a time-dependent term $\frac{\sigma e^{-\iota t}}{3N^{2}K_{\mu}\sqrt{q}}$ is deployed. The introduction of the auxiliary variables $\beta_{i}$ and $\gamma_{i}$ offers two main advantages. First, the triggering interval in \eqref{triggerc} can be flexibly tuned through the evolution of these auxiliary variables, allowing it to adapt to varying system conditions. In contrast, the triggering interval in \cite{yu2024distributed} is solely determined by the exponentially decaying term, resulting in a fixed and rapidly diminishing triggering period.
Second, since both the system information and auxiliary variables in \eqref{triggerc} are functions of the consensus errors, the proposed mechanism naturally adapts to the evolution of the system state. Conversely, the time-dependent term in \cite{yu2024distributed} depends only on time $t$, which limits its adaptability. Moreover, the event-triggered mechanism \eqref{triggerc} adopts two independent triggering conditions for the two interactive variables, which provides greater flexibility compared with traditional dynamic mechanisms employing a single triggering condition. In most previous studies \cite{xu2021event,Yang2024,li2023distributed,liu2023predefined}, multiple transmission variables share one common triggering inequality, such as the one in \cite{li2023distributed}: $$\begin{aligned}
t_{l+1}^i=&\inf \{t>t_l^i \mid 2(\overline{\mathbf{e}}^i)^\top \overline{\mathbf{e}}^i+e_1^{i\top} e_1^i+e_2^{i\top} e_2^i+e_3^{i\top} e_3^i\\
&-\sigma^i(\|\sum_{\hat{j}=1}^N a_{i \hat{j}}(\hat{y}^i-\hat{y}^{\hat{j}})\|^2+\|\sum_{\hat{j}=1}^N a_{i \hat{j}}(\hat{u}_i-\hat{u}_{\hat{j}})\|^2 \\
&+\|\sum_{k=1}^{m_i} \sum_{l \in(\mathcal{N}_{i k}^i \cup\{k\})} a_{j l}^i(\hat{g}_{j k}^i-\hat{g}_{l k}^i)\|^2)>\frac{\eta^i}{\chi^i}\},
\end{aligned}$$where multiple variables ($\overline{\mathbf{e}}^i=\hat{\mathbf{g}}^i-\mathbf{g}^i$, $e_1^i=\hat{y}^i-y^i$, $e_2^i=\hat{u}_i-u_i$, $e_3^i=\hat{v}_i-v_i$) share a single triggering boundary. In such a setup, all associated errors accumulate together, often resulting in more frequent triggering. In contrast, the event-triggered mechanism \eqref{triggerc} assigns each variable its own triggering condition, thus enhancing flexibility and effectively reducing communication frequency.
\end{rem}
\begin{thm}\label{dingli222}
Suppose that Assumptions \ref{jiashe11} and \ref{jiashe22} hold. For any $T\geqslant0$, when $x(0)\in\Omega$, $\mu(0)=\mathbf{0}_{Nq}$, $m+\frac{m_{0}}{2}>l$, $k>\frac{2(l-l_{0})^{2}+(8m+4m_{0})l_{0}}{(4m+2m_{0}-4l)\lambda_{2}(\mathbf{L})}$ and $K_\mu\geqslant NK_g$, based on the VGNE seeking algorithm \eqref{algo2} under dynamic event-triggered mechanism \eqref{triggerc}, the following regret and fit bounds hold
\begin{equation*}
	\left\{\begin{aligned}
		\mathcal{R}^{\top}\leqslant& \frac{1}{2}\left\|\Upsilon(0)-\Upsilon^{*}\right\|^{2}+\frac{K_{\mu}}{4\eta}\sum_{i=1}^{N}\gamma_{i}(0)+\frac{k}{8\eta}\sum_{i=1}^{N}\beta_{i}(0),\\
		\mathcal{F}^{\top}\leqslant&\sqrt{N}\Vert\Upsilon(0)-\Upsilon^{*}\Vert+2N\sqrt{K_{f}T}\\
 	&+\sqrt{\frac{K_{\mu}N}{2\eta}\sum_{i=1}^{N}\gamma_{i}(0)+\frac{kN}{4\eta}\sum_{i=1}^{N}\beta_{i}(0)}.
	\end{aligned}
	\right.
\end{equation*}
In other words, we have $\mathcal{R}^{\top}=\mathcal{O}(1)$ and $\mathcal{F}^{\top}=\mathcal{O}(\sqrt{T})$.
\end{thm}
\begin{proof}
Select the candidate Lyapunov function as follows $$\widetilde{V}=\frac{1}{2}\left\|\Upsilon-\Upsilon^{*}\right\|^{2}+\frac{1}{2}\left\|\mu-\mu^{*}\right\|^{2}.$$
Then, taking the time derivative of $\widetilde{V}$ along the direction of \eqref{algo2} and based on Lemma \ref{tou22}, it yields that
\begin{small}\begin{equation*}\label{2V1}
	\begin{aligned} \dot{\widetilde{V}}\leqslant&\sum_{i=1}^{N}(x_{i}-x_{i}^{*})^{\top}[-\partial_{x_{i}}J_{i}(t,\Upsilon^{i})-kR_{i}\sum_{j\in\mathcal{N}_{i}}(e_{\Upsilon}^{i}-e_{\Upsilon}^{j}\\
		&+\Upsilon^{i}-\Upsilon^{j})-\mu_{i}^{\top}\partial_{x_i} g_{i}(t,x_{i})]+\sum_{i=1}^{N}(\Upsilon_{-i}^{i}-\Upsilon_{-i}^{i^{*}})^{\top}\\
		&(-kS_{i}\sum_{j\in\mathcal{N}_{i}}(e_{\Upsilon}^{i}-e_{\Upsilon}^{j}+\Upsilon^{i}-\Upsilon^{j}))+\sum_{i=1}^{N}(\mu_{i}-\mu_i^*)^{\top}g_{i}(t,x_{i})\\
		&-2K_{\mu}\sum_{i=1}^{N}(\mu_{i}-\mu_i^*)^{\top}\sum_{j\in\mathcal{N}_{i}}\text{sgn}(e_{\mu}^{i}-e_{\mu}^{j}+\mu_{i}-\mu_{j}).
	\end{aligned}
\end{equation*}\end{small}

\textbf{Step 1}: By letting $\bar{\mu}^{*}=\mathbf{0}_{q}$ and $K_\mu\geqslant NK_g$, similar to the proof in Theorem \ref{dingli11}, it can be obtained that
\begin{equation*}\label{18}
	\begin{aligned}
		\dot{\widetilde{V}}\leqslant&-\bar{\mathcal{R}}^{\top}-\frac{k}{2}(\Upsilon-\Upsilon^{*})^{\top}\mathbf{L}(\Upsilon-\Upsilon^{*})\\
		&-k(\Upsilon-\Upsilon^{*})^{\top}\mathbf{L}e_{\Upsilon}+\frac{K_{\mu}}{2}\sum_{i=1}^{N}\sum_{j=1}^{N}a_{ij}\Vert\mu_{i}-\mu_{j}\Vert_{1}\\
			&-K_{\mu}\sum_{i=1}^{N}\sum_{j=1}^{N}a_{ij}(\mu_{i}-\mu_{j})^{\top}\text{sgn}(\hat{\mu}_{i}-\hat{\mu}_{j}).
	\end{aligned}
	\end{equation*}
To facilitate the subsequent analysis, deploy $\widetilde{V}_{1}$ and $\widetilde{V}_{2}$ to represent the terms in the above formula with
\begin{equation*}
	\widetilde{V}_{1}=-\frac{k}{2}(\Upsilon-\Upsilon^{*})^{\top}\mathbf{L}(\Upsilon-\Upsilon^{*})-k(\Upsilon-\Upsilon^{*})^{\top}\mathbf{L}e_{\Upsilon},
\end{equation*}
and
\begin{equation*}
	\begin{aligned}
	\widetilde{V}_{2}=&\frac{K_{\mu}}{2}\sum_{i=1}^{N}\sum_{j=1}^{N}a_{ij}\Vert\mu_{i}-\mu_{j}\Vert_{1}\\
			&-K_{\mu}\sum_{i=1}^{N}\sum_{j=1}^{N}a_{ij}(\mu_{i}-\mu_{j})^{\top}\text{sgn}(\hat{\mu}_{i}-\hat{\mu}_{j}).
	\end{aligned}
\end{equation*}
	\par
According to the symmetry of the Laplacian matrix and the positive semidefinite property of $\mathcal{L}=\mathcal{D}-\mathcal{A}$ and $\mathcal{D}+\mathcal{A}$, it can be derived that
	\begin{equation*}\label{det}
		\begin{aligned}
			\widetilde{V}_{1}&=-\frac{k}{2}\left( (\Upsilon-\Upsilon^{*})^{\top}\mathbf{L}(\Upsilon-\Upsilon^{*})+2(\Upsilon-\Upsilon^{*})^{\top}\mathbf{L}e_{\Upsilon}\right)\\
			&=-\frac{k}{2}(\hat{\Upsilon}^{\top}\mathbf{L}\hat{\Upsilon}-e_{\Upsilon}^{\top}\mathbf{L}e_{\Upsilon})\\
			&\leqslant-\frac{k}{4}\left(\sum_{i=1}^{N}(\sum_{j=1}^{N}a_{ij}\Vert \hat{\Upsilon}^{i}-\hat{\Upsilon}^{j}\Vert^{2}-4d_{i}\Vert e_{\Upsilon}^{i}\Vert^{2})\right).
		\end{aligned}
	\end{equation*}
	Then, according to the dynamic triggering condition \eqref{triggerc}, it can be further obtained that
	\begin{equation*}\label{circle2}
		\widetilde{V}_{1}\leqslant \frac{k}{8}\left(\sum_{i=1}^{N}\beta_{i}-\sum_{i=1}^{N}(\sum_{j=1}^{N}a_{ij}\Vert\hat{\Upsilon}^{i}-\hat{\Upsilon}^{j}\Vert^{2}-4d_{i}\Vert e_{\Upsilon}^{i}\Vert^{2})\right).
	\end{equation*}
	\par
	Also, by noticing the inequality $\Vert\upsilon\Vert_{1}\leqslant\sqrt{q}\Vert\upsilon\Vert, \forall\upsilon\in\mathbb{R}^{q}$, it derives that
	\begin{equation}\label{circle3}
		\begin{aligned} \widetilde{V}_{2}\leqslant&\frac{K_{\mu}}{2}\sum_{i=1}^{N}\sum_{j=1}^{N}a_{ij}\Vert\hat{\mu}_{i}-\hat{\mu}_{j}\Vert_{1}+\frac{K_{\mu}\sqrt{q}}{2}\sum_{i=1}^{N}\sum_{j=1}^{N}\Vert e_{\mu}^{i}-e_{\mu}^{j}\Vert\\
			&-K_{\mu}\sum_{i=1}^{N}\sum_{j=1}^{N}a_{ij}\Vert\hat{\mu}_{i}-\hat{\mu}_{j}\Vert_{1}+K_{\mu}\sqrt{q}\sum_{i=1}^{N}\sum_{j=1}^{N}\Vert e_{\mu}^{i}-e_{\mu}^{j}\Vert.
		\end{aligned}
	\end{equation}
	In addition, the following inequality holds
	\begin{equation}\label{rs}
		\sum_{i=1}^{N}\sum_{j=1}^{N}\Vert e_{\mu}^{i}-e_{\mu}^{j}\Vert\leqslant 2N\sum_{i=1}^{N}\Vert e_{\mu}^{i}\Vert.
	\end{equation}
Taking \eqref{rs} into \eqref{circle3}, it can be obtained that
\begin{equation*}\label{circle33}
	\begin{aligned}
		\widetilde{V}_{2}\leqslant&3K_{\mu}\sqrt{q}N\sum_{i=1}^{N}\Vert e_{\mu}^{i}\Vert-\frac{K_{\mu}}{2}\sum_{i=1}^{N}\sum_{j=1}^{N}a_{ij}\Vert\hat{\mu}_{i}-\hat{\mu}_{j}\Vert_{1}\\
		=&\frac{K_{\mu}}{2}\sum_{i=1}^{N}\left(6\sqrt{q}N\Vert e_{\mu}^{i}\Vert-\sum_{j=1}^{N}a_{ij}\Vert\hat{\mu}_{i}-\hat{\mu}_{j}\Vert_{1}\right)\\
		\leqslant&\frac{K_{\mu}}{4}\sum_{i=1}^{N}\left(6\sqrt{q}N\Vert e_{\mu}^{i}\Vert-\sum_{j=1}^{N}a_{ij}\Vert\hat{\mu}_{i}-\hat{\mu}_{j}\Vert_{1}\right)\\
		&+\frac{K_{\mu}}{4}\sum_{i=1}^{N}\gamma_{i},
	\end{aligned}
\end{equation*}
 where the last inequality holds due to the dynamic triggering condition \eqref{triggerc}.
 \par
Letting $$\widetilde{\mathbb{V}}=\widetilde{V}+\frac{K_{\mu}}{4\eta}\sum_{i=1}^{N}\gamma_{i}(t)+\frac{k}{8\eta}\sum_{i=1}^{N}\beta_{i}(t),$$it yields that
 \begin{equation}\label{regretb}
 	\dot{\widetilde{\mathbb{V}}}\leqslant-\bar{\mathcal{R}}^{\top}.
 \end{equation}
 By integrating time $t$ from $0$ to $T$ on both sides of \eqref{regretb}, it yields that
 \begin{equation*}
 \mathcal{R}^{\top}\leqslant\frac{1}{2}\left\|\Upsilon(0)-\Upsilon^{*}\right\|^{2}+\frac{K_{\mu}}{4\eta}\sum_{i=1}^{N}\gamma_{i}(0)+\frac{k}{8\eta}\sum_{i=1}^{N}\beta_{i}(0).
\end{equation*}
 Thus it can be concluded that $\mathcal{R}^{\top}=\mathcal{O}(1)$.
 \par
\textbf{Step 2}: By letting \begin{equation*}
		\bar{\mu}^*_i=
		\begin{cases}
			0,&\text{if }\mathcal{F}_{i}^{\top}\leqslant0,\\
			\frac{\mathcal{F}_{i}^{\top}}{N},&\text{if }\mathcal{F}_{i}^{\top}>0,
		\end{cases}
	\end{equation*}similar to the proof in Theorem \ref{dingli11}, it can be obtained that
\begin{equation*}
	\dot{\widetilde{\mathbb{V}}}\leqslant-\bar{\mathcal{R}}^{\top}-\bar{\mu}^{*\top}\sum_{i=1}^{N}g_{i}(t,x_{i}).
	\end{equation*}
 Then, similarly, it gives that
 \begin{equation*}
 	\begin{aligned}
 	\frac{1}{N}\Vert \mathcal{F}\Vert^{2}\leqslant&\frac{1}{2}\left\|\Upsilon(0)-\Upsilon^{*}\right\|^{2}+\frac{K_{\mu}}{4\eta}\sum_{i=1}^{N}\gamma_{i}(0)+\frac{k}{8\eta}\sum_{i=1}^{N}\beta_{i}(0)\\
 	&+\frac{1}{2N}\Vert\mathcal{F}\Vert^{2}+2NK_{f}T.
 	\end{aligned}
 \end{equation*}
 Finally, by transposition and scaling, it can be derived that
 \begin{equation*}
 	\begin{aligned}
\mathcal{F}^{\top}\leqslant&\sqrt{N}\Vert\Upsilon(0)-\Upsilon^{*}\Vert+\sqrt{\frac{K_{\mu}N}{2\eta}\sum_{i=1}^{N}\gamma_{i}(0)+\frac{kN}{4\eta}\sum_{i=1}^{N}\beta_{i}(0)}\\
 	&+2N\sqrt{K_{f}T}.
 	\end{aligned}
 \end{equation*}
 Thus, $\mathcal{F}^{\top}=\mathcal{O}(\sqrt{T})$ can be reached.
 \end{proof}
 \par
 During the triggering process, an extreme scenario might exist in which there are infinite number of triggers within a finite amount of time. In real-life applications, this phenomenon, normally known as Zeno behavior, ought to be strictly forbidden. The following theorem is given to demonstrate that the triggering condition \eqref{triggerc} will not exhibit Zeno behavior.
 \begin{thm}
No player's action profile under the dynamic event-triggered mechanism \eqref{triggerc} will display Zeno behavior.
 \end{thm}
\begin{proof}
The theorem is proved by contradiction. Suppose that the dynamic event-triggered mechanism \eqref{triggerc} exhibits Zeno behavior, i.e., there exists a constant $t_{\infty}^{i}>0$ such that $\lim_{k\rightarrow\infty}t_{k}^{i}=t_{\infty}^{i}$. For $t\in[t_{k}^{i},t_{k+1}^{i})$, by separating the two inequalities in \eqref{triggerc}, we obtain \begin{equation*}
	t_{k+1}^{\Upsilon_i}=\inf_{t>t_{k}^{i}}\{t\mid4d_{i}\Vert e_{\Upsilon}^{i}\Vert ^{2}>\sum_{j=1}^{N}a_{ij}\Vert\Upsilon^{i}(t_{k}^{i})-\Upsilon^{j}(t_{k}^{j})\Vert^{2}+\beta_{i}\},
\end{equation*} 
\begin{equation*}
t_{k+1}^{\mu_i}=\inf_{t>t_{k}^{i}}\{t\mid6\sqrt{q}N\Vert e_{\mu}^{i}\Vert>\sum_{j=1}^{N}a_{ij}\Vert\mu_{i}(t_{k}^{i})-\mu_{j}(t_{k}^{j})\Vert_1+\gamma_{i}\}.
\end{equation*} Hence, the overall triggering instant is given by $t_{k+1}^i=\min\{t_{k+1}^{\Upsilon_i},t_{k+1}^{\mu_i}\}$.

For $t_{k+1}^{\Upsilon_i}$, the time derivative of the measurement error $e_{\Upsilon}^{i}$ satisfies
\begin{equation*}
\Vert\dot{e}_{\Upsilon}^{i}\Vert=\Vert\dot{\Upsilon}^{i}\Vert\leqslant M,
\end{equation*}
where $M>0$ denotes the upper bound of the closed set $\Omega_i$. Consequently,
\begin{equation*}
\Vert e_{\Upsilon}^{i}\Vert=\Vert\int_{t_{k}^{i}}^{t}\dot{e}_{\Upsilon}^{i}(s)\text{d}s\Vert\leqslant\int_{t_{k}^{i}}^{t}\Vert\dot{e}_{\Upsilon}^{i}(s)\Vert \text{d}s\leqslant M(t-t_{k}^{i}).
\end{equation*}
Similarly, for $t_{k+1}^{\mu_i}$, the time derivative of $e_{\mu}^{i}$ is bounded as
\begin{equation*}\label{zb}
\Vert\dot{e}_{\mu}^{i}\Vert=\Vert\dot{\mu}_{i}\Vert\leqslant \Vert g_{i}(t,x_{i})-2K_{\mu}\sum_{j\in\mathcal{N}_{i}}\text{sgn}(\hat{\mu}_{i}-\hat{\mu}_{j})\Vert\leqslant\tilde{M},
\end{equation*}
where the first inequality follows from $\left\|\Pi_{S}[x, v]\right\| \leqslant\|v\|$ (Remark 2.1 in \cite{zhang1995stability}) and the second inequality is given according to Remark \ref{youjie}. Then, one derives
\begin{equation*}
\Vert e_{\mu}^{i}\Vert=\Vert\int_{t_{k}^{i}}^{t}\dot{e}_{\mu}^{i}(s)\text{d}s\Vert\leqslant\int_{t_{k}^{i}}^{t}\Vert\dot{e}_{\mu}^{i}(s)\Vert \text{d}s\leqslant \tilde{M}(t-t_{k}^{i}).
\end{equation*}

Then, if $t_{k+1}^i=t_{k+1}^{\Upsilon_i}$, it follows that
\begin{equation*}
\begin{aligned}
	M(t_{k+1}^{\Upsilon_i}-t_{k}^{i})&\geqslant\Vert e_{\Upsilon}^{i}(t_{k+1}^{\Upsilon_i})\Vert\\
	&>\sqrt{\frac{\sum_{j=1}^{N}a_{ij}\Vert\Upsilon^{i}(t_{k}^{i})-\Upsilon^{j}(t_{k}^{j})\Vert^{2}+\beta_{i}(t_{k+1}^{\Upsilon_i})}{4d_{i}}}\\
	&\geqslant\sqrt{\frac{\beta_{i}(t_{k+1}^{\Upsilon_i})}{4d_{i}}}\geqslant\sqrt{\frac{\beta_i(0)e^{-2\eta t_{\infty}^{i}}}{4d_{i}}}>0.
\end{aligned}
\end{equation*}
If $t_{k+1}^i=t_{k+1}^{\mu_i}$, one obtains
\begin{equation*}
\begin{aligned}
	\tilde{M}(t_{k+1}^{\mu_i}-t_{k}^{i})&\geqslant\Vert e_{\mu}^{i}(t_{k+1}^{\mu_i})\Vert\\
	&>\frac{\sum_{j=1}^{N}a_{ij}\Vert\mu_{i}(t_{k}^{i})-\mu_{j}(t_{k}^{j})\Vert_1+\gamma_{i}(t_{k+1}^{\mu_i})}{6\sqrt{q}N}\\
	&\geqslant\frac{\gamma_{i}(t_{k+1}^{\mu_i})}{6\sqrt{q}N}\geqslant\frac{\gamma_i(0)e^{-2\eta t_{\infty}^{i}}}{6\sqrt{q}N}>0.
\end{aligned}
\end{equation*}
Therefore, the inter-event interval satisfies
\begin{equation*}
	t_{k+1}^i-t_k^i>\min\left\{\sqrt{\frac{\beta_i(0)e^{-2\eta t_{\infty}^{i}}}{4M^2d_{i}}},\frac{\gamma_i(0)e^{-2\eta t_{\infty}^{i}}}{6\tilde{M}\sqrt{q}N}\right\}>0.
\end{equation*}
Thus, $\lim_{k\rightarrow\infty}t_{k}^{i}=+\infty$, which contradicts the assumption $\lim_{k\rightarrow\infty}t_{k}^{i}=t_{\infty}^{i}$. Hence, Zeno behavior is strictly excluded.
\end{proof}
\subsection{Noise-Resilient VGNE Seeking Algorithm}
For VGNE seeking algorithm \eqref{algo1}, the signum functions can be interpreted as representing the relative directions of the interactive variables, since $\operatorname{sgn}\left(\mu_i-\mu_j\right)=\operatorname{sgn}\left(\frac{\mu_i-\mu_j}{\left\|\mu_i-\mu_j\right\|}\right)$. In practice, however, the measurement of the relative directions between neighboring players may be corrupted by communication noise or sensor imperfections. To investigate the effect of such inaccuracies on the direction of relative state measurements, we replace $\operatorname{sgn}\left(\mu_i-\mu_j\right)$ by $\operatorname{sgn}\left(\mu_i-\mu_j+\zeta_{i j}(t)\left\|\mu_i-\mu_j\right\|\right)$, where $\zeta_{i j}(t)\in\mathbb{R}^q$ denotes a time-varying communication noise from player $j$ to player $i$. It is noteworthy that the noise $\zeta_{i j}(t)$ arises when player $i$ measures the direction of the relative state of its neighbor $j$, and thus it is added to the normalized relative state.

Following the definitions in Subsection \ref{refit}, the regret and fit in the expected sense are given by
$$ \begin{aligned} \mathfrak{R}^\top &=\mathbb{E}\left[\int_{0}^{T}\sum_{i=1}^{N}\left(J_{i}(t,x_{i}(t),x_{-i}^{*})-J_{i}(t,x_{i}^{*},x_{-i}^{*})\right)\text{d}t\right], \\ \mathfrak{F}^\top &=\left\|\mathbb{E}\left[\Big[\int_{0}^{T}\sum_{i=1}^{N}g_{i}(t,x_{i})\text{d}t\Big]_{+}\right]\right\|.\end{aligned} $$
\begin{assum}\label{jiashenoise}
The measurement noises $\zeta_{i j}(t)$ are independent and identically distributed random variables with a symmetric common probability density function.
\end{assum}

As a result, a continuous-time noise-resilient VGNE seeking algorithm is proposed as follows
\begin{equation}\label{algonoise}
	\left\{
	\begin{aligned}
		\dot{x}_{i}  = &\Pi_{\Omega_{i}}[x_{i},-\partial_{x_{i}}J_{i}(t,\Upsilon^{i})-\mu_{i}^{\top}\partial_{x_i} g_{i}(t,x_{i})\\
&-kR_{i}\sum_{j\in\mathcal{N}_{i}}(\Upsilon^{i}-\Upsilon^{j})],\\
		\dot{\Upsilon}_{-i}^{i}  =& -kS_{i}\sum_{j\in\mathcal{N}_{i}}(\Upsilon^{i}-\Upsilon^{j}),\\
		\dot{\mu}_{i}  =& \Pi_{\mathbb{R}_{\geqslant0}^{q}}[\mu_{i},-K_{\mu}\sum_{j\in\mathcal{N}_{i}}\text{sgn}(\mu_{i}-\mu_{j}+\zeta_{i j}(t)\left\|\mu_i-\mu_j\right\|)\\
&+g_{i}(t,x_{i})].
	\end{aligned}
	\right.
	\end{equation}
\begin{thm}\label{dinglinoise}
Suppose that Assumptions \ref{jiashe11}, \ref{jiashe22} and \ref{jiashenoise} hold. For any $T\geqslant0$, when $x(0)\in\Omega$, $\mu(0)=\mathbf{0}_{Nq}$, $m+\frac{m_{0}}{2}>l$, $k>\frac{(l-l_{0})^{2}+(4m+2m_{0})l_{0}}{(4m+2m_{0}-4l)\lambda_{2}(\mathbf{L})}$, $K_\mu\geqslant N^2K_g$, and $\mathbb{E}[\|\zeta_{i j}(t)\|_1] \leqslant 1-\frac{N^2 K_g}{K_\mu}$, based on the VGNE seeking algorithm \eqref{algonoise}, the following regret and fit bounds hold
\begin{equation*}
	\left\{\begin{aligned}
		\mathfrak{R}^{\top}&\leqslant \frac{1}{2}\Vert\Upsilon(0)-\Upsilon^{*}\Vert^{2},\\
		\mathfrak{F}^{\top}&\leqslant 2N\sqrt{K_{f}T}+\sqrt{N}\Vert \Upsilon(0)-\Upsilon^{*}\Vert.
	\end{aligned}
	\right.
\end{equation*}
In other words, we have $\mathfrak{R}^{\top}=\mathcal{O}(1)$ and $\mathfrak{F}^{\top}=\mathcal{O}(\sqrt{T})$.
\end{thm}
\begin{proof}
Select the candidate Lyapunov function as follows $$\mathds{V}(t)=\frac{1}{2}\left\|\Upsilon-\Upsilon^{*}\right\|^{2}+\frac{1}{2}\left\|\mu-\mu^*\right\|^{2}.$$Then, taking the time derivative of $\mathds{V}$ along the direction of \eqref{algonoise} and based on Lemma \ref{tou22}, it yields that
\begin{equation*}
\begin{aligned}
\dot{\mathds{V}}\leqslant&\sum_{i=1}^{N}(x_{i}-x_{i}^{*})^{\top}\left[-\partial_{x_{i}}J_{i}(t,\Upsilon^{i})-kR_{i}\sum_{j=1}^{N}a_{ij}(\Upsilon^{i}-\Upsilon^{j})\right]\\
&-\sum_{i=1}^{N}(x_{i}-x_{i}^{*})^{\top}\mu_{i}^{\top}\partial_{x_i} g_{i}(t,x_{i})+\sum_{i=1}^{N}(\mu_{i}-\mu_{i}^{*})^{\top}g_{i}(t,x_{i})\\
&-K_{\mu}\sum_{i=1}^{N}\sum_{j\in\mathcal{N}_{i}}(\mu_{i}-\mu_{i}^{*})^{\top}\text{sgn}(\mu_{i}-\mu_{j}+\zeta_{i j}\left\|\mu_i-\mu_j\right\|)\\
&+\sum_{i=1}^{N}(\Upsilon_{-i}^{i}-\Upsilon_{-i}^{i^{*}})^{\top}\left[-kS_{i}\sum_{j=1}^{N}a_{ij}(\Upsilon^{i}-\Upsilon^{j})\right].
\end{aligned}
\end{equation*}
\par
Similar to the analysis in Theorem \ref{dingli11}, deploy $\mathds{V}_{1}$ and $\mathds{V}_{2}$ to represent the terms in the above formula with
\begin{equation*}
\begin{aligned}
	\mathds{V}_{1}=&\sum_{i=1}^{N}(x_{i}^{*}-x_{i})^{\top}\left[\partial_{x_{i}}J_{i}(t,\Upsilon^{i})+kR_{i}\sum_{j=1}^{N}a_{ij}(\Upsilon^{i}-\Upsilon^{j})\right]\\
	&+\sum_{i=1}^{N}(\Upsilon_{-i}^{i}-\Upsilon_{-i}^{i^{*}})\left[-kS_{i}\sum_{j=1}^{N}a_{ij}(\Upsilon^{i}-\Upsilon^{j})\right],
\end{aligned}
\end{equation*}
and
\begin{equation*}
	\begin{aligned}
	\mathds{V}_{2}=&\sum_{i=1}^{N}(x_{i}^{*}-x_{i})^{\top}\mu_{i}^{\top}\partial_{x_i} g_{i}(t,x_{i})+\sum_{i=1}^{N}(\mu_{i}-\mu_{i}^{*})^{\top}g_{i}(t,x_{i})\\
&-K_{\mu}\sum_{i=1}^{N}\sum_{j\in\mathcal{N}_{i}}(\mu_{i}-\mu_{i}^{*})^{\top}\text{sgn}(\mu_{i}-\mu_{j}+\zeta_{i j}\left\|\mu_i-\mu_j\right\|).
	\end{aligned}
\end{equation*}
Following the analysis in Theorem \ref{dingli11}, $\mathds{V}_{1}$ can be scaled as 
\begin{equation*}
	\mathds{V}_{1}\leqslant-\bar{\mathcal{R}}^{\top}-\lambda_{\min}(\mathbf{M})\Vert\Upsilon-\Upsilon^{*}\Vert^{2}.
\end{equation*}

As for $\mathds{V}_{2}$, it can be similarly derived that
	\begin{align}\label{dvi2}
		\mathds{V}_{2}\leqslant&\sum_{i=1}^{N}g_{i}(t,x_{i}^{*})^{\top}\mu_{i}-\sum_{i=1}^{N}g_{i}(t,x_{i})^{\top}\mu_{i}^{*}\\
		&-K_{\mu}\sum_{i=1}^{N}\sum_{j\in\mathcal{N}_{i}}(\mu_{i}-\mu_{i}^{*})^{\top}\text{sgn}(\mu_{i}-\mu_{j}+\zeta_{i j}\left\|\mu_i-\mu_j\right\|)\nonumber.
	\end{align}

\textbf{Step 1}: By letting $\bar{\mu}^{*}=\mathbf{0}_{q}$ and taking expectations on both sides of \eqref{dvi2}, it gives that \begin{align}\label{dv202}
		&\mathbb{E}[\mathds{V}_{2}|\mu(t)]\\
\leqslant&\mathbb{E}[K_g \sum_{i=1}^N \sum_{j=1}^N\left\|\mu_i-\mu_j\right\||\mu(t)]\nonumber\\
		&-\frac{K_{\mu}}{2}\mathbb{E}[\sum_{i=1}^{N}\sum_{j=1}^Na_{ij}\|\mu_{i}-\mu_{j}+\zeta_{i j}\|\mu_i-\mu_j\|\|_1|\mu(t)]\nonumber,
	\end{align}
where the inequality is obtained based on the symmetry of the probability density function of $\zeta_{ij}(t)$. Then, taking expectations on both sides of \eqref{dv202}, one derives that 
$$\begin{aligned}
\mathbb{E}[\mathds{V}_{2}]\leqslant &\mathbb{E}[K_g \sum_{i=1}^N \sum_{j=1}^N\left\|\mu_i-\mu_j\right\|]\\
&-\frac{K_{\mu}}{2}(1-\mathbb{E}\left[\left\|\zeta_{i j}\right\|_1\right]) \sum_{i=1}^N \sum_{j=1}^N a_{i j} \mathbb{E}\left[\left\|\mu_{i}-\mu_{j}\right\|\right] \\
\leqslant &\mathbb{E}[K_g \sum_{i=1}^N \sum_{j=1}^N\left\|\mu_i-\mu_j\right\|]\\
&-K_{\mu}(1-\mathbb{E}\left[\left\|\zeta_{i j}\right\|_1\right]) \mathbb{E}\left[\left\|\mu_{i_0}-\mu_{j_0}\right\|\right] \\
\leqslant &\mathbb{E}[K_g \sum_{i=1}^N \sum_{j=1}^N\left\|\mu_i-\mu_j\right\|]\\
&-\frac{K_{\mu}}{N^2}(1-\mathbb{E}\left[\left\|\zeta_{i j}\right\|_1\right]) \sum_{i=1}^N \sum_{j=1}^N \mathbb{E}\left[\left\|\mu_i-\mu_j\right\|\right],
\end{aligned}$$
where the second inequality holds based on the connectivity of the undirected graph $\mathcal{G}$ for $(i_0,j_0)=\text{argmax}_{(i,j)}\|\mu_{i}-\mu_{j}\|$.
\par
According to the condition $\mathbb{E}[\|\zeta_{i j}(t)\|_1] \leqslant 1-\frac{N^2 K_g}{K_\mu}$, gathering the analysis for $\mathds{V}_{1}$ and $\mathds{V}_{2}$, it holds that
\begin{equation*}
	\mathfrak{R}^{\top}\leqslant\frac{1}{2}\left\|\Upsilon(0)-\Upsilon^{*}\right\|^{2}.
\end{equation*}
Thus it can be concluded that $\mathfrak{R}^{\top}=\mathcal{O}(1)$.
\par
\textbf{Step 2}: Similar to the proof in Theorem \ref{dingli11}, the fit bound can be obtained that
\begin{equation*}
	\mathfrak{F}^{\top}\leqslant 2N\sqrt{K_{f}T}+\sqrt{N}\Vert\Upsilon(0)-\Upsilon^{*}\Vert.
\end{equation*}
Thus, $\mathfrak{F}^{\top}=\mathcal{O}(\sqrt{T})$ can be reached.
\end{proof}
\begin{rem}
As far as we know, this work is the first to explicitly incorporate measurement noises into the VGNE seeking process for online games. This consideration is crucial since measurement uncertainties are inevitable in practical systems and may significantly affect convergence accuracy. Unlike \cite{cao2021decentralized,yu2024distributed}, which focus solely on online optimization, this paper handles competitive multi-player interactions with coupling costs and constraints under noisy conditions.
\end{rem}
 \section{Numerical Simulation}\label{NE}
\subsection{Online UAV Swarm Game}
A UAV swarm, analogous to a bee colony \cite{Wang2023}, consists of multiple UAVs capable of dynamically sensing and sharing information to execute tasks such as resource exploration, intelligence reconnaissance, and offensive-defensive operations. In this subsection, a numerical example including five UAVs interacting over an undirected connected graph in Fig. \ref{topology} is given to testify the effectiveness of algorithms \eqref{algo1}, \eqref{algonoise} and \eqref{algo2} with dynamic event-triggered mechanism \eqref{triggerc}, with $x_i=\col(x_{i,a},x_{i,b})~ (i=1,2,\cdots,5)$ denoting UAV $i$'s coordinates. The time-varying function $J_{i}$ represents the objective of each UAV to reach its target position while maintaining an appropriate distance from neighboring UAVs to ensure connectivity. Meanwhile, the time-varying function $g_{i}$ characterizes the coupling constraints on the position coordinates of the UAVs, and they are formulated as
 \begin{equation*}
 \begin{aligned}
 J_{1}=&1.11((x_{1,a}-2\cos(10t)-1)^{2}+(x_{1,b}-2\cos(10t)-1)^{2})\\&+0.01\sum_{j\in\mathcal{N}_{1}}\Vert x_{1}-x_{j}\Vert^{2},\\
J_{2}=&1.12((x_{2,a}-\cos(20t)-1)^{2}+(x_{2,b}-\cos(20t)-1)^{2})\\&+0.02\sum_{j\in\mathcal{N}_{2}}\Vert x_{2}-x_{j}\Vert^{2},\\
 J_{3}=&1.13((x_{3,a}-\cos(20t)-3)^{2}+(x_{3,b}-\cos(20t)-3)^{2})\\&+0.03\sum_{j\in\mathcal{N}_{3}}\Vert x_{3}-x_{j}\Vert^{2},\\
 J_{4}=&1.14((x_{4,a}-3\cos(10t)-2)^{2}+(x_{4,b}-3\cos(10t)-2)^{2})\\&+0.04\sum_{j\in\mathcal{N}_{4}}\Vert x_{4}-x_{j}\Vert^{2},\\
 J_{5}=&1.15((x_{5,a}-\cos(15t)-2)^{2}+(x_{5,b}-\cos(15t)-2)^{2})\\&+0.04\sum_{j\in\mathcal{N}_{5}}\Vert x_{5}-x_{j}\Vert^{2},
 \end{aligned}
 \end{equation*}
  \begin{equation*}
 \begin{aligned}
 g_{1}=&(0.3\sin(15t)+1.7)x_{1,a}+(0.2\sin(10t)+1.8)x_{1,b}-1,\\
 g_{2}=&(0.4\sin(20t)+1.6)x_{2,a}+(0.4\sin(20t)+1.6)x_{2,b}-2,\\
 g_{3}=&(0.5\sin(10t)+1.5)x_{3,a}+(0.6\sin(25t)+1.4)x_{3,b}-3,\\
 g_{4}=&(0.6\sin(15t)+1.4)x_{4,a}+(0.8\sin(15t)+1.2)x_{4,b}-4,\\
 g_{5}=&(0.7\sin(10t)+1.3)x_{5,a}+(0.5\sin(10t)+1.5)x_{5,b}-5.\\
 \end{aligned}
 \end{equation*}
 \begin{figure}[htbp]
 	\centering
 	\includegraphics[width=0.45\linewidth]{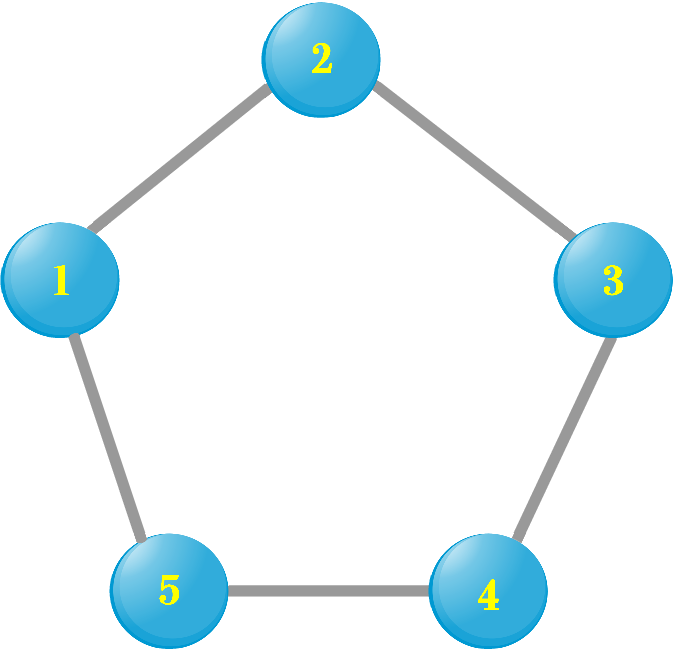}
 	\caption{The communication topology among five UAVs.}
 	\label{topology}
 \end{figure}Moreover, the searching space of each UAV is restricted to a closed convex set $\Omega_{i}$, which is defined as the closed interval constraint $[-3,3]\times[-3,3]$.

First of all, set the step size as $0.001s$ and $T=100s$. The initial values $x_i(0)$ are randomly selected in $[-3,3]\times[-3,3]$, and let $\mu(0)=\mathbf{0}_5,~ k=20$ and $K_\mu=500$. Then, based on the cost functions, constraints and communication graph, the problem constants can be calculated as $m=2.22,~m_0=2.32,~l=2.7,~K_g=17,~\lambda_{2}(\mathbf{L})=1.3820$. Substituting these values into the conditions of Theorems \ref{dingli11} and \ref{dingli222} yields $m+\frac{m_{0}}{2}=2.22+1.16=3.38>l=2.7$, $k>9.711\geqslant\frac{(l-l_{0})^{2}+(4m+2m_{0})l_{0}}{(4m+2m_{0}-4l)\lambda_{2}(\mathbf{L})}$, $k>19.422\geqslant\frac{2(l-l_{0})^{2}+(8m+4m_{0})l_{0}}{(4m+2m_{0}-4l)\lambda_{2}(\mathbf{L})}$ and $K_\mu\geqslant NK_g=85$. Therefore, the chosen simulation parameters satisfy the inequalities required by Theorems \ref{dingli11} and \ref{dingli222}. Based on the continuous-time VGNE seeking algorithm \eqref{algo1}, the evolutions of $\mathcal{R}^\top$ and $\frac{\mathcal{F}^\top}{\sqrt{T}}$ are shown in Fig. \ref{regret1}. It is clear that the regret bound $\mathcal{R}$ and the fit bound $\mathcal{F}$ are $\mathcal{O}(1)$ and $\mathcal{O}(\sqrt{T})$, respectively.
 \begin{figure}[htbp]
 	\centering
 	\includegraphics[width=\linewidth]{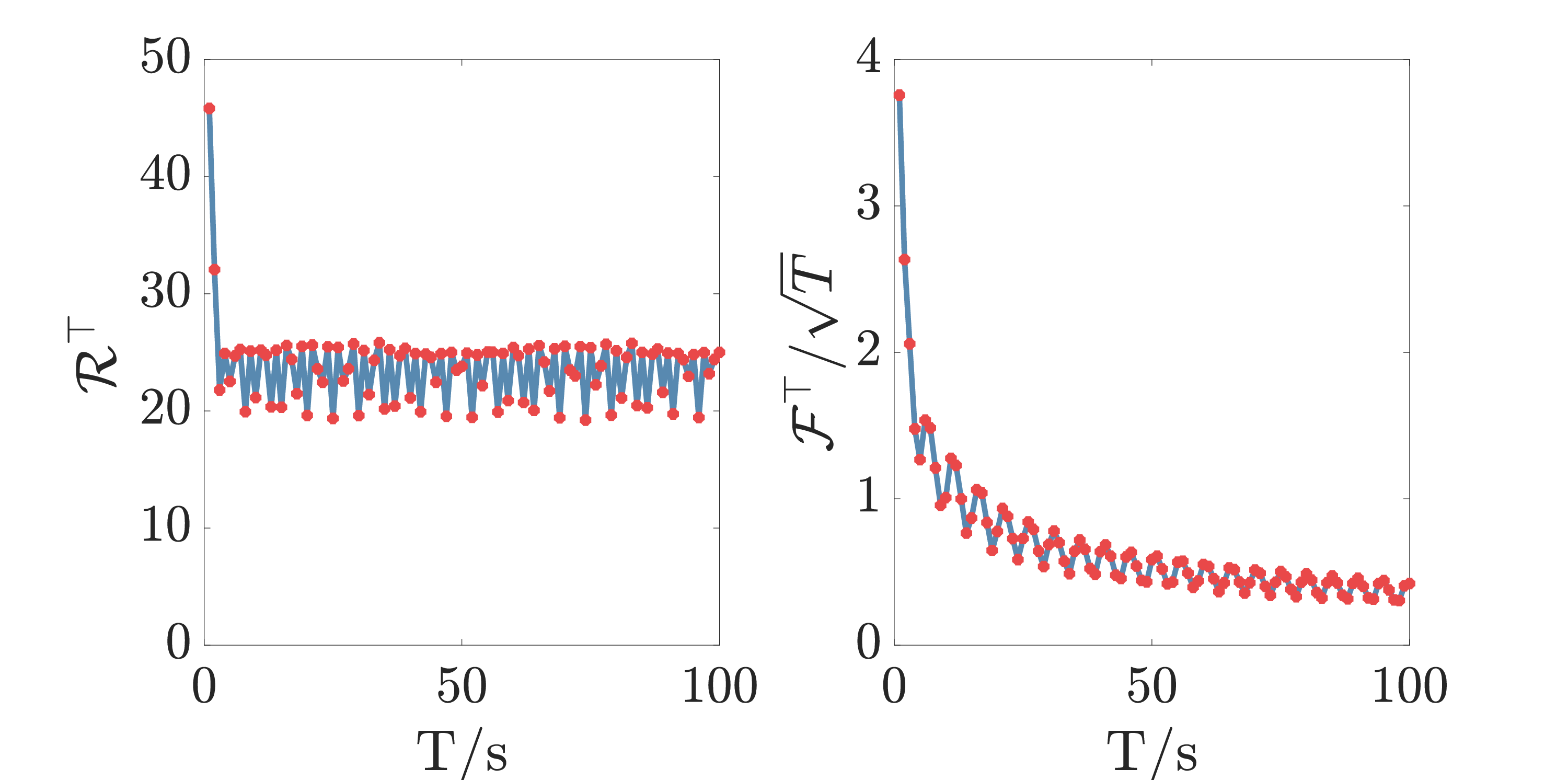}
 	\caption{Evolutions of $\mathcal{R}^\top$ and $\frac{\mathcal{F}^\top}{\sqrt{T}}$ based on algorithm \eqref{algo1}.}
 	\label{regret1}
 \end{figure}

Then, set the initial values $\beta_i(0)=\gamma_i(0)=1000$ and $\eta=0.001$. Based on the VGNE seeking algorithm \eqref{algo2} with dynamic event-triggered mechanism \eqref{triggerc}, the evolutions of $\mathcal{R}^\top$ and $\frac{\mathcal{F}^\top}{\sqrt{T}}$ are shown in Fig. \ref{regret2}.\begin{figure}[htbp]
 	\centering
 	\includegraphics[width=\linewidth]{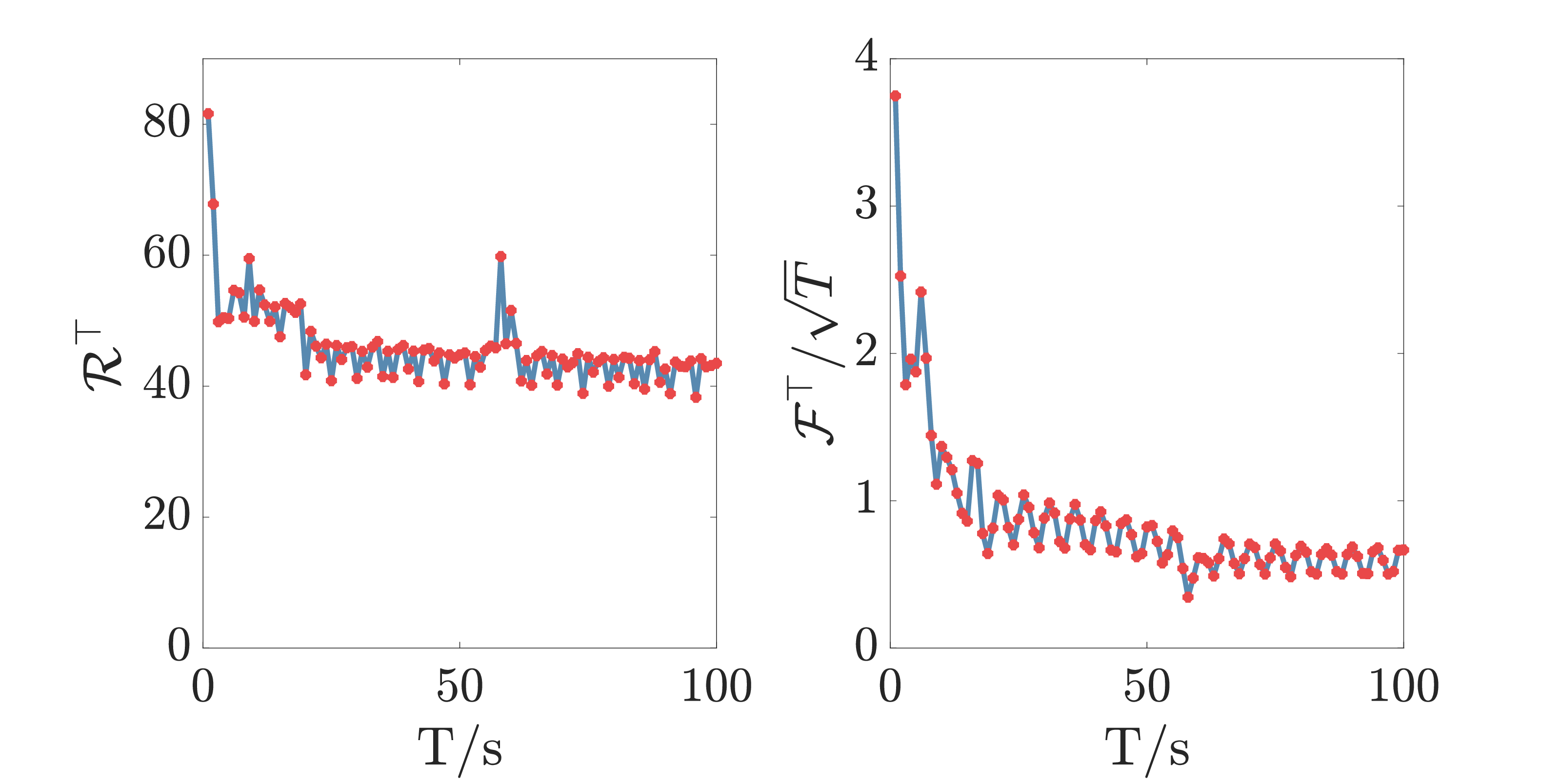}
 	\caption{Evolutions of $\mathcal{R}^\top$ and $\frac{\mathcal{F}^\top}{\sqrt{T}}$ based on algorithm \eqref{algo2} with dynamic event-triggered mechanism \eqref{triggerc}.}
 	\label{regret2}
 \end{figure} The constant regret bound and the sublinear fit bound are still remained. For comparison, the evolutions of the fits based on algorithms \eqref{algo1} and \eqref{algo2} are also included in Fig. \ref{fit1}. It can be seen that the trajectories of the fits fluctuate violently and there is no trend of convergence. Moreover, to testify the efficiency of the event-triggered mechanism \eqref{triggerc}, the inter-event intervals of UAVs 1-2 are shown in Fig. \ref{tri}. It evidently illustrates that the event-triggered mechanism can reduce unnecessary communication among UAVs.
  \begin{figure}[htbp]
 	\centering
 	\includegraphics[width=\linewidth]{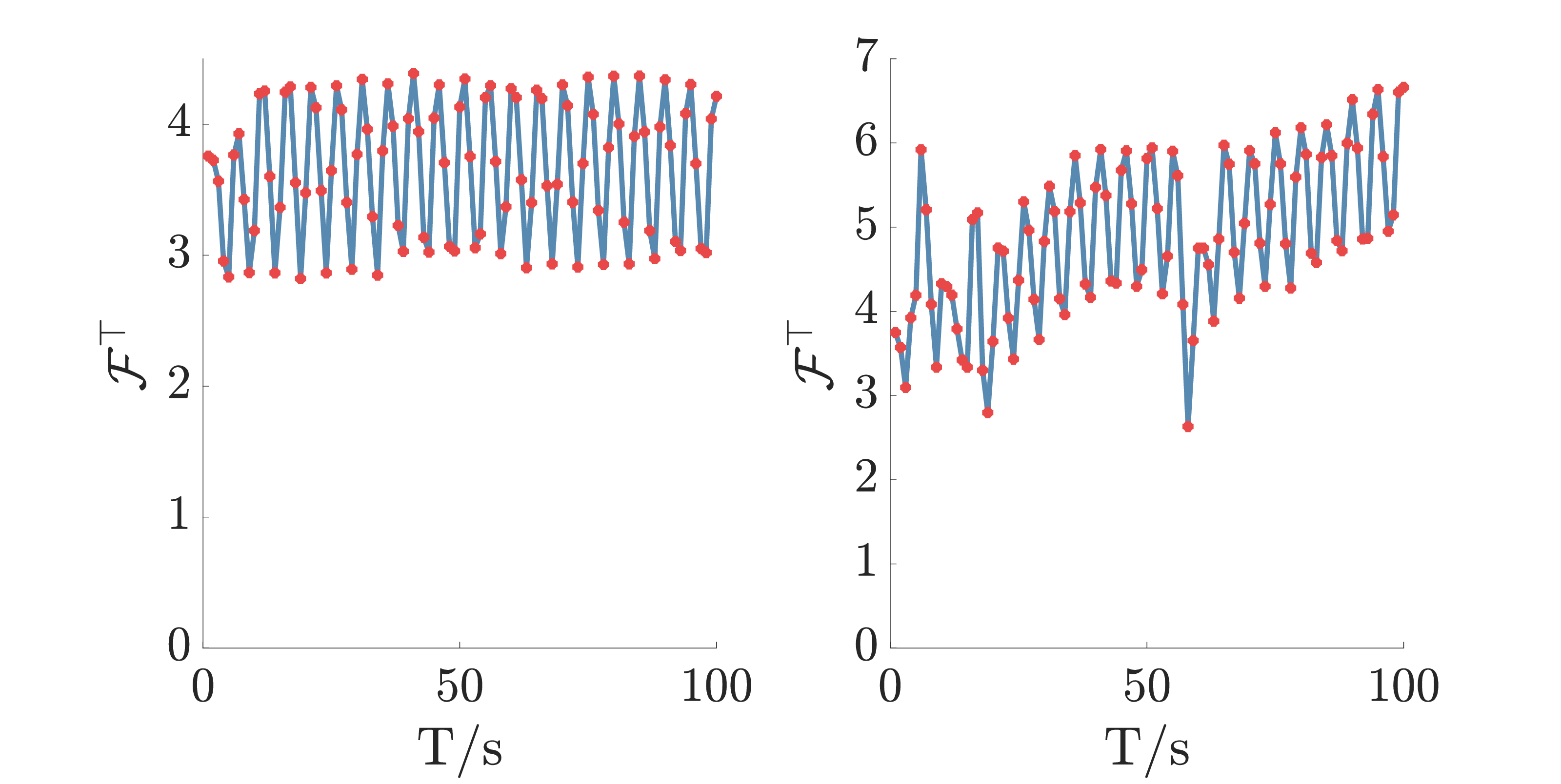}
 	\caption{Evolutions of $\mathcal{F}^\top$ based on algorithms \eqref{algo1} (Left) and \eqref{algo2} (Right).}
 	\label{fit1}
 \end{figure}
  \begin{figure}[htbp]
 	\centering
 	\includegraphics[width=\linewidth]{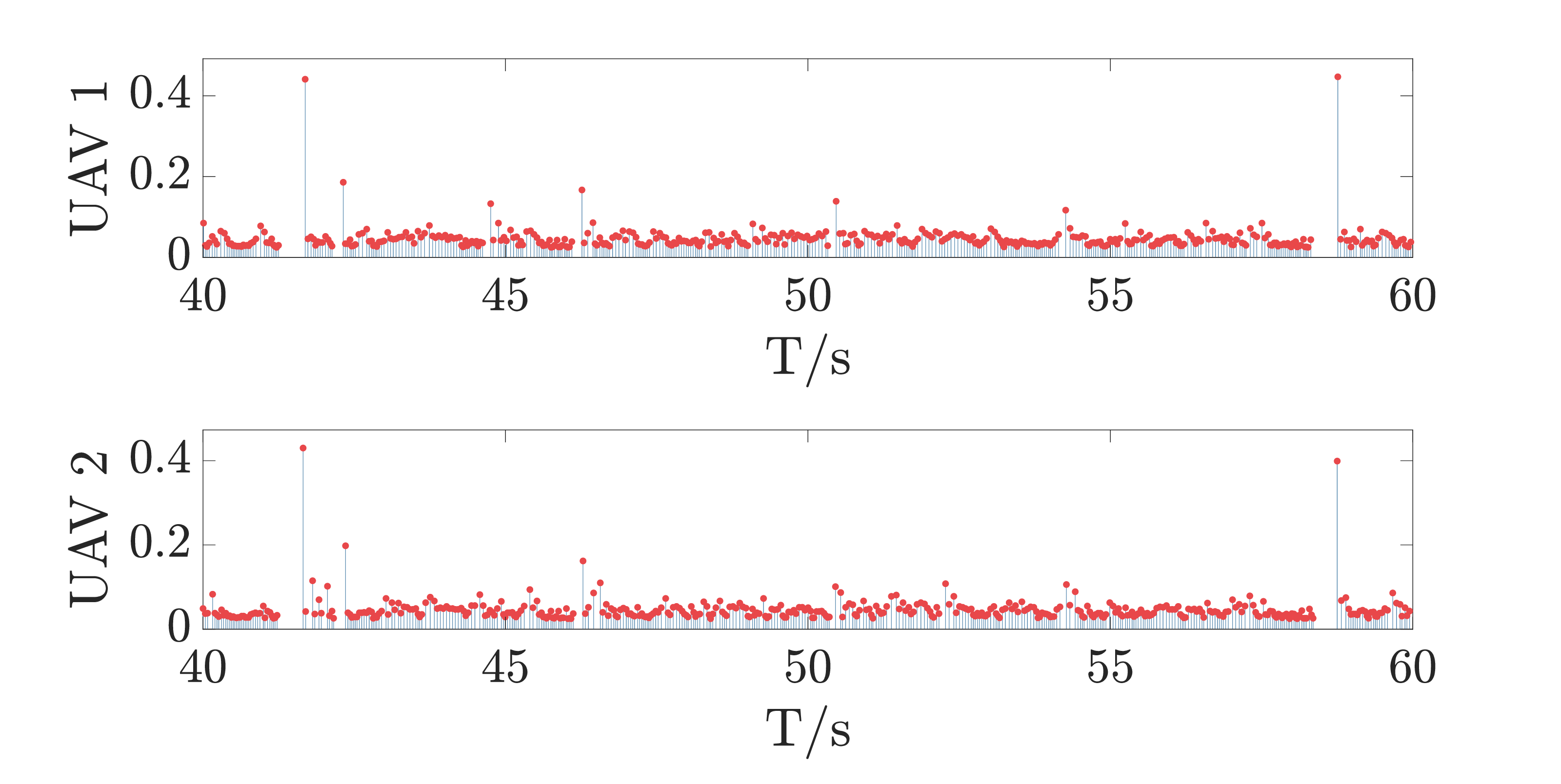}
 	\caption{Inter-event intervals of UAVs 1-2.}
 	\label{tri}
 \end{figure}
 
 Finally, the effectiveness of the noise-resilient VGNE seeking algorithm \eqref{algonoise} is validated, where the measurement noise follows a Gaussian distribution with the probability density function $$\varpi_{ij}(y)=\frac{1}{\nu\sqrt{2\pi}}e^{-\frac{(y-\kappa)^2}{2\nu^2}},$$where $\kappa=0$ and $\nu=0.118$. It can be verified that $K_\mu\geqslant N^2K_g=425$ and $\mathbb{E}[\|\zeta_{i j}(t)\|_1] \leqslant 1-\frac{N^2 K_g}{K_\mu}=0.15$. Thus, the selected parameters satisfy the inequalities required by Theorem \ref{dinglinoise}. Based on algorithm \eqref{algonoise}, the evolutions of $\mathcal{R}^\top$ and $\frac{\mathcal{F}^\top}{\sqrt{T}}$ are illustrated in Fig. \ref{noisetu}. It can be observed that the regret bound $\mathcal{R}$ and the fit bound $\mathcal{F}$ scale as $\mathcal{O}(1)$ and $\mathcal{O}(\sqrt{T})$, respectively.
 \begin{figure}[htbp]
 	\centering
 	\includegraphics[width=\linewidth]{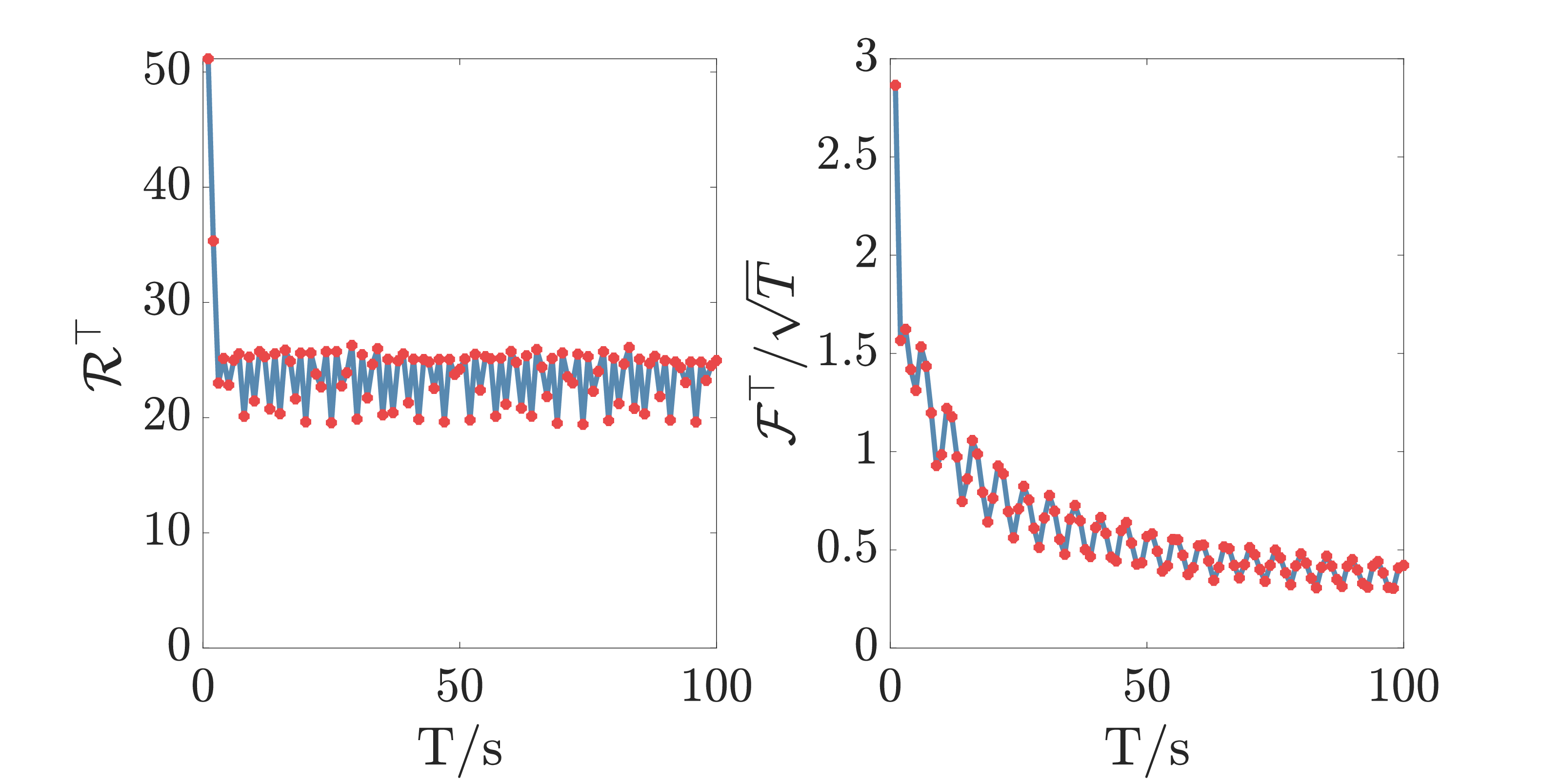}
 	\caption{Evolutions of $\mathcal{R}^\top$ and $\frac{\mathcal{F}^\top}{\sqrt{T}}$ based on algorithm \eqref{algonoise}.}
 	\label{noisetu}
 \end{figure}
 \subsection{Online Nash-Cournot Game}
 In this subsection, an online Nash-Cournot game involving ten firms is introduced to testify the effectiveness of algorithm \eqref{algo_juhe} for online aggregative game. The online Nash-Cournot game is represented by $\bar{\Gamma}(\mathcal{V},J_{i},\Omega_{i}\cap U(t))$, where the closed convex set constraints $\Omega=\Omega_1\times\Omega_2\cdots\times\Omega_{10}$ and the coupling inequality constraint $U(t)$ denote the production constraints and market capacity constraints on the strategies of all the firms, respectively. $J_{i}$ is the local time-varying cost function of firm $i$, including an aggregative term $\sigma(x)$. Each firm is willing to interact with its neighbors to grasp necessary information via an undirected connected communication graph as shown in Fig. \ref{topology2}. The production quantity of the $i$th firm is denoted by $x_{i}\in\mathbb{R}$. In addition, due to dynamic economic factors, the demand price, marginal costs, and production cost are time-varying, where the $i$th firm's time-varying production cost and time-varying demand price are provided as $\rho_{i}(t,x_{i})=\alpha_{i}(t)x_{i}^2$ and $\tau_{i}(t,x)=\chi_{i}(t)-\sigma(x)$ for some $\alpha_{i}(t),~\chi_{i}(t)>0$. To conclude, the cost function of firm $i$ is given as
\begin{equation*}
	J_{i}(t,x_{i},\sigma(x))=\rho_{i}(t,x_{i})-x_{i}\tau_{i}(t,x),~\forall t\in[0,T].
\end{equation*}
Generally, firm $i$ aims to develop a strategy (i.e., its production quantity $x_{i}$) to reduce its overall cost over $T$ time span. It is preset that $\Omega_{1}=\Omega_{2}=\cdots=\Omega_{10}=[0,30]$, the firm $i$'s production cost $\rho_{i}(t,x_{i})=x_{i}^2(0.1\sin(20t)+2)$ and the demand price $\tau_{i}(t,x)=21+i/9-0.5i\sin(20t)-\sigma(x)$ with $\sigma(x)=\sum_{i=1}^{N}x_{i}/10$. Additionally, the coupling inequality constraint, which serves as the time-varying market capacity constraint, is designed as $\sum_{i=1}^{N}5x_{i}\leqslant\sum_{i=1}^{N}l_{i}(t)$ with $l_{i}(t)=5*(4+\sin(10t))$.
\par
First, similar to the online UAV swarm game, set the step size as $0.001s$ and $T=100s$. The initial values $x_{i}(0)$ are randomly chosen within the interval $[0,30]$, and let $\mu(0)=\mathbf{0}_{10}$, $w=280$ and $K_{\mu}=1400$. Then, based on the cost functions and constraints, the problem constants can be calculated as $m=3.9,~m_0=4,~l\leqslant5.63,~K_g=135$. Substituting these values into the conditions of Corollary \ref{juhe_v} yields $m+\frac{m_{0}}{2}=3.9+2=5.9\geqslant l$, $K_\mu\geqslant NK_g=1350$ and $w>(N-1)\iota$. Finally, the evolutions of $\mathcal{R}^{\top}$ and $\frac{\mathcal{F}^{\top}}{\sqrt{T}}$ based on algorithm \eqref{algo_juhe} are shown in Fig. \ref{nc_1}. It is verified that the regret bound $\mathcal{R}$ and the fit bound $\mathcal{F}$ are $\mathcal{O}(1)$ and $\mathcal{O}(\sqrt{T})$, respectively.
\begin{figure}[htbp]
	\centering
	\includegraphics[width=\linewidth]{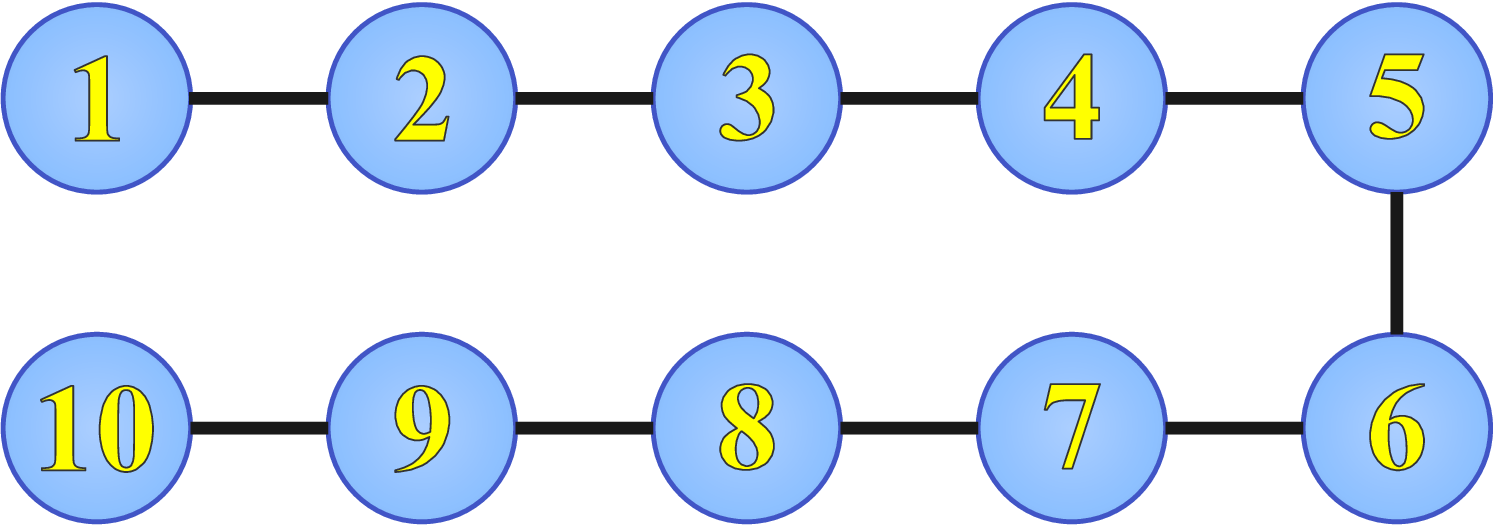}
	\caption{The communication topology among ten firms.}
	\label{topology2}
\end{figure}
\begin{figure}[htbp]
	\centering
	\includegraphics[width=\linewidth]{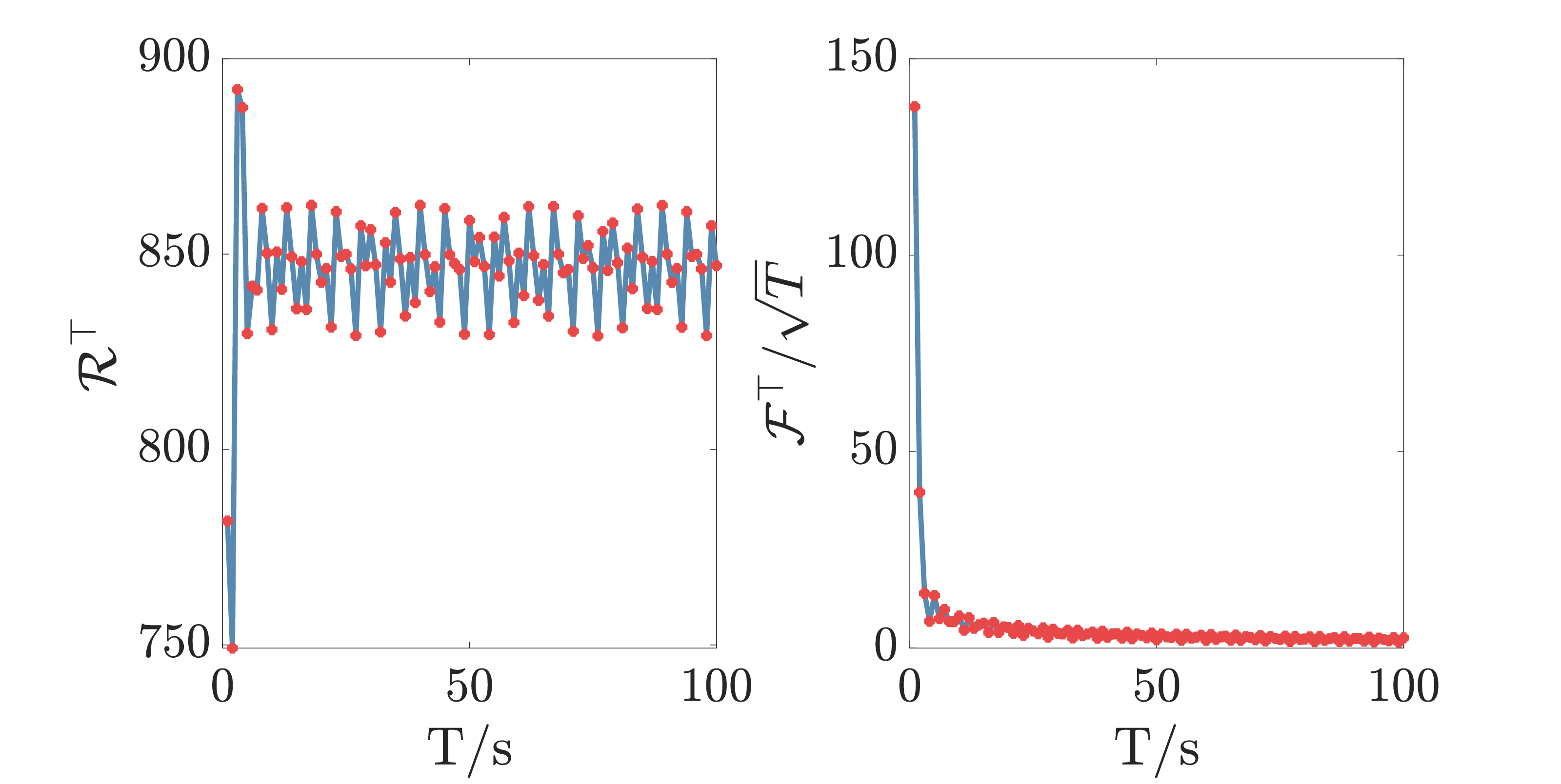}
	\caption{Evolutions of $\mathcal{R}^\top$ and $\frac{\mathcal{F}^\top}{\sqrt{T}}$ based on algorithm \eqref{algo_juhe}.}
	\label{nc_1}
\end{figure}
\section{Conclusion}\label{C}
In this paper, two novel continuous-time distributed VGNE seeking algorithm tailored for online noncooperative game and online aggregative game with time-varying coupling inequality constraints are proposed. A constant regret bound and a sublinear fit bound are successfully achieved, which are superior to the basic criteria established for online optimization problems and online games. Moreover, a dynamic event-triggered mechanism is introduced into the VGNE seeking algorithm to reduce communication overhead among players, while still maintaining the desired performance metrics and prohibiting Zeno behavior. Moreover, the proposed algorithm is further extended to account for measurement noise and it is shown that comparable performance can still be guaranteed under moderate noise disturbances. The numerical examples presented further validate the effectiveness of the proposed algorithms, demonstrating their applicability in solving complex game-theoretic problems under time-varying constraints. In the future, the authors will focus on the robustness of the VGNE seeking algorithm for online game considering the influence of external disturbances.
\section{References}
\bibliography{reference}
\bibliographystyle{IEEEtran}
\begin{IEEEbiography}[{\includegraphics[width=\linewidth]{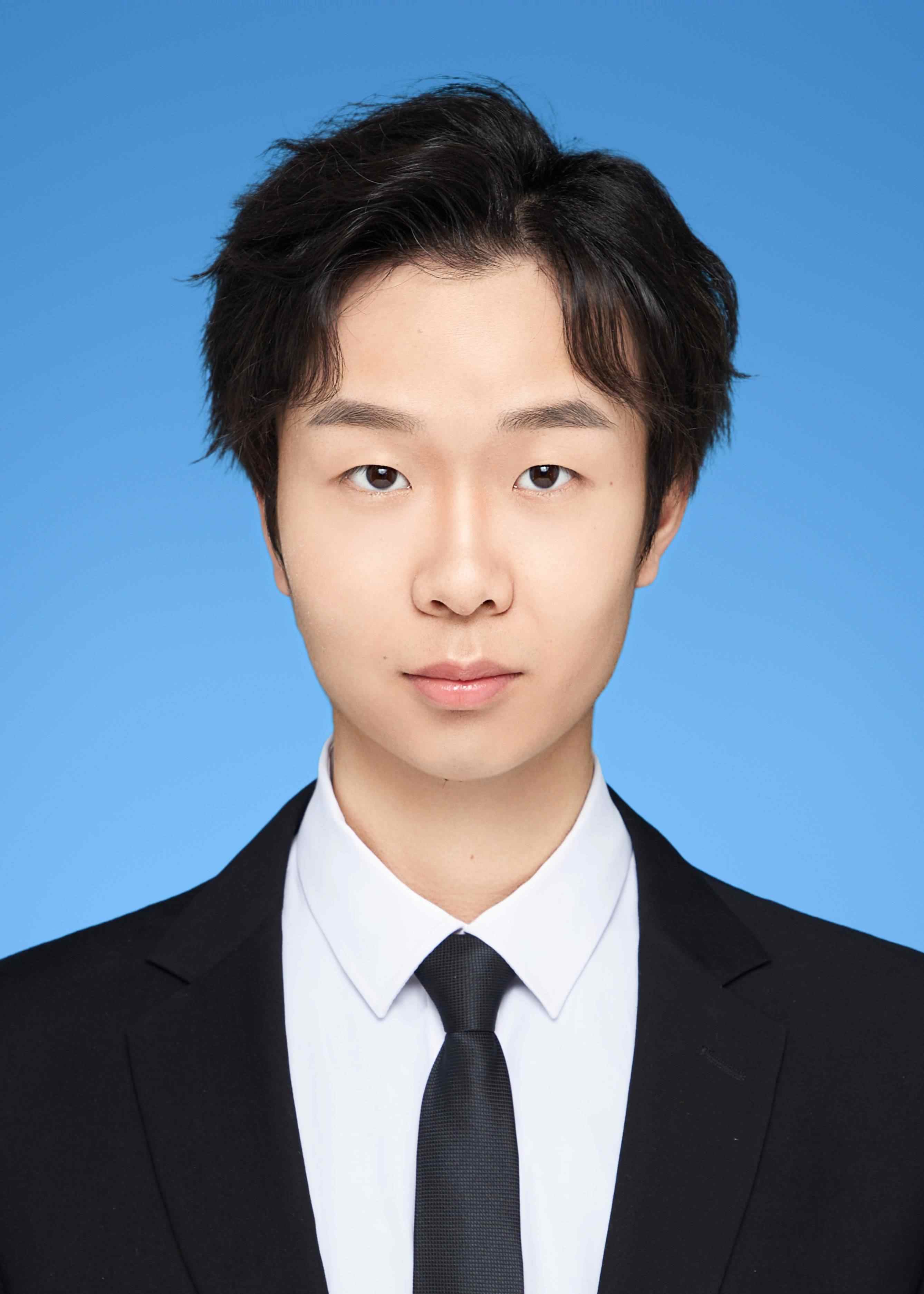}}]{Jianing Chen}
was born in 1999. He received his B.S. degree in applied mathematics from Harbin Institute of Technology, Weihai, China, in 2021, where he is currently pursuing the Ph.D. degree with the Department of Mathematics. Since 2023, he has been a Joint Ph.D. student at the Department of Systems Engineering, City University of Hong Kong. His current research interests include game theory and distributed optimization.
\end{IEEEbiography}
\begin{IEEEbiography}[{\includegraphics[width=\linewidth]{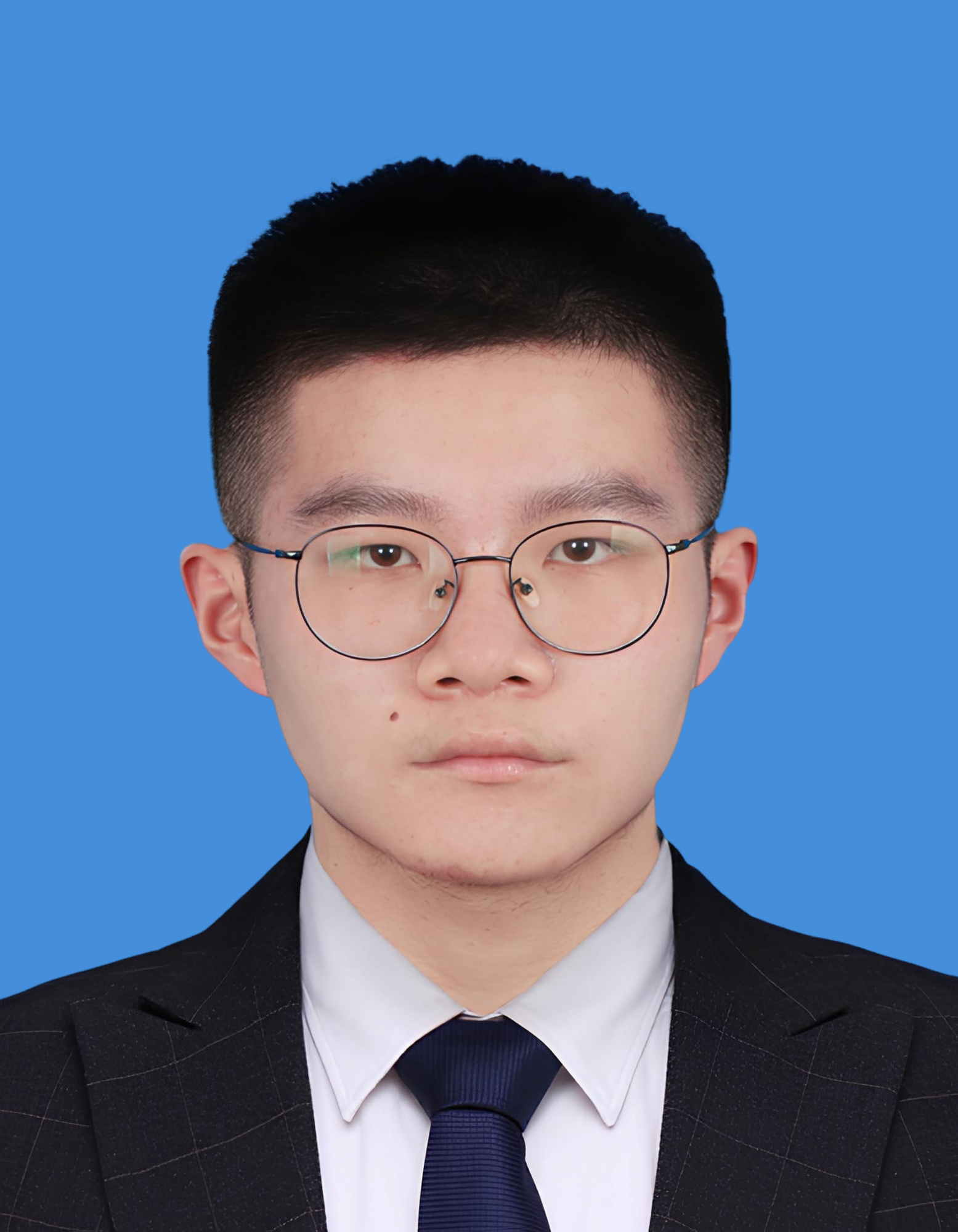}}]{Sichen Qian}
was born in 2002. He received his B.S. degree in applied mathematics from Harbin Institute of Technology, Weihai, China, in 2024. He is currently
pursuing the M.S. degree with Southeast University, Nanjing, China. His current research interests include game theory.
\end{IEEEbiography}
\begin{IEEEbiography}[{\includegraphics[width=\linewidth]{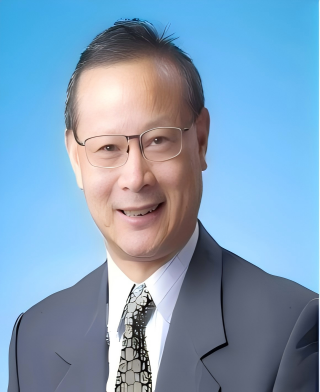}}]{Chuangyin Dang}
(Senior Member, IEEE) received the Ph.D. degree (cum laude) in operations research/mathematical economics from Tilburg University, Tilburg, The Netherlands, in 1991. He held faculty positions with the University of California at Davis, Davis, CA, USA; Delft University of Technology, Delft, The Netherlands; and The University of Auckland, Auckland, New Zealand. He was a Research Fellow with the Cowles Foundation for Research in Economics, Yale University, New Haven, CT, USA, where he was invited to work as a Research Fellow, in 1994. He is currently a Professor of systems engineering with the City University of Hong Kong, Hong Kong. He is best known for the developments of the D1-triangulation of the Euclidean space and simplicial and fixed-point iterative methods for integer programming. He has published over 130 articles in top journals. His research focuses on systems modeling, analysis, and optimization. Dr. Dang received the Award of Outstanding Research Achievements from Tilburg University in 1990 for his significant contributions.
\end{IEEEbiography}
\begin{IEEEbiography}[{\includegraphics[width=\linewidth]{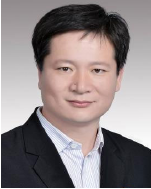}}]{Sitian Qin}
was born in Shandong, China, in 1981. He received the Ph.D. degree in mathematics from the Harbin Institute of Technology, Harbin, China, in 2010. Since 2010, he has worked at Harbin Institute of Technology, Weihai, China. He was an associate professor from 2013 to 2018 and is currently a professor. His current research interests include game theory and neural network theory.
\end{IEEEbiography}
\end{document}